\documentclass[12pt]{amsart}

\usepackage{verbatim, amssymb, enumitem, mathtools}

\usepackage[breaklinks=true,colorlinks=true,linkcolor=blue,citecolor=red,urlcolor=blue,psdextra,pdfencoding=auto]{hyperref}
\allowdisplaybreaks

\usepackage{orcidlink}

\renewcommand \a{\alpha}
\renewcommand \b{\beta}

\newcommand \la{\lambda}

\newcommand \id{\mathrm{id}}
\newcommand \br{\mathbb{R}}
\newcommand \bc{\mathbb{C}}
\newcommand \bh{\mathbb{H}}
\newcommand \bo{\mathbb{O}}

\newcommand \Span{\operatorname{Span}}

\newcommand \db{\partial}
\newcommand \SO{\mathrm{SO}}
\newcommand \Sp{\mathrm{Sp}}
\newcommand \SU{\mathrm{SU}}
\newcommand \SL{\mathrm{SL}}
\newcommand \Ff{\mathrm{F}_4}
\newcommand \Spin{\mathrm{Spin}}
\newcommand \rJ{\mathrm{J}}
\newcommand \rL{\mathrm{L}}
\newcommand \rR{\mathrm{R}}

\newcommand \cP{\mathcal{P}}

\newcommand \cV{\mathcal{V}}

\newcommand \cH{\mathcal{H}}
\newcommand \cL{\mathcal{L}}

\newcommand \cK{\mathcal{K}}

\newcommand \cS{\mathcal{S}}
\newcommand \cQ{\mathcal{Q}}
\newcommand \ocS{\overline{\mathcal{S}}}

\newcommand \sK{\mathsf{K}}

\newcommand \sQ{\mathsf{Q}}

\newcommand\g{\mathfrak g}

\newcommand \so{\mathfrak{so}}
\newcommand \spg{\mathfrak{sp}}

\newcommand \ug{\mathfrak{u}}

\newcommand \ri{\mathrm{i}}

\newcommand \Sym{\operatorname{Sym}}

\newcommand \diag{\operatorname{diag}}

\newcommand{\on}{\overline{\nabla}}
\newcommand{\od}{\overline{\delta}}
\newcommand{\oi}{\overline{\iota}}
\newcommand{\ot}{\overline{\tau}}
\newcommand{\orJ}{\overline{\rJ}}

\newcommand \<{\langle}
\renewcommand \>{\rangle}
\newcommand \ip{\<\cdot,\cdot\>}

\newcommand \GL{\mathrm{GL}}

\newtheorem{theorem}{Theorem}
\newtheorem*{theorem*}{Theorem}

\newtheorem*{corollary*}{Corollary}
\newtheorem*{conj*}{Conjecture}
\newtheorem{lemma}{Lemma}
\newtheorem{proposition}{Proposition}
\newtheorem*{prop*}{Proposition}

\theoremstyle{definition}

\newtheorem*{definition*}{Definition}

\theoremstyle{remark}
\newtheorem{remark}{Remark}

\newtheorem*{notation*}{Notation}
\newtheorem*{algorithm*}{Algorithm}
\newtheorem*{example*}{Example}

\begin{document}

\title{Killing tensors on projective spaces} 

\author{Owen Dearricott \orcidlink{0000-0002-2984-9230}}
\address{Department of Mathematical and Physical Sciences, La Trobe University, Melbourne, Victoria, 3086, Australia}
\email{O.Dearricott@latrobe.edu.au}

\author{Yuri Nikolayevsky \orcidlink{0000-0002-9528-1882}}
\address{Department of Mathematical and Physical Sciences, La Trobe University, Melbourne, Victoria, 3086, Australia}
\email{y.nikolayevsky@latrobe.edu.au}

\thanks {The authors were supported by ARC Discovery grant DP210100951.}

\subjclass{53C35, 53B20}

\keywords{Killing tensor field, quaternionic projective space}

\begin{abstract}
A Killing tensor field on a Riemannian space corresponds to an integral of the geodesic flow polynomial in momenta. A (contravariant) Killing tensor field is called \emph{decomposable} if it is a polynomial in Killing vector fields. While all Killing tensor fields on the spaces of constant curvature and on the complex projective space are decomposable, there is an explicitly constructed family of indecomposable quadratic Killing tensor fields on the quaternionic projective spaces $\bh P^n, \, n \ge 3$. We prove that the algebra of Killing tensor fields on the quaternionic projective space is generated by Killing vector fields and these indecomposable quadratic Killing tensor fields. We also give another proof of the fact that the algebra of Killing tensor fields on the complex projective space is generated by Killing vector fields.
\end{abstract}

\maketitle

\section{Introduction}
\label{s:intro}

In the past few years, a substantial progress has been made towards understanding Killing tensor fields on Riemannian symmetric spaces~\cite{E, MN1, MN2, EL, MN3, MNN}.

There are three equivalent definitions of Killing tensor fields. Namely, a \emph{covariant Killing tensor field} $K=K(x)_{i_1\dots i_d}$ of rank $d \ge 1$ on a Riemannian manifold $(M,ds^2=g_{ij}dx^i dx^j)$ is a covariant symmetric tensor field that satisfies the Killing equation
\begin{equation}\label{eq:defK}
  K_{(i_1\dots i_d,j)}=0,
\end{equation}
where the comma denotes the covariant derivative and the parentheses denote the symmetrisation by all indices. An equivalent definition is that the function $\xi\in T_xM \mapsto K(x)_{i_1\dots i_d} \xi^{i_1} \cdots \xi^{i_d}$ polynomial in the velocities is an integral of the geodesic flow of $(M,ds^2)$: for any naturally parameterised geodesic $s \mapsto \gamma(s)$ of $(M,ds^2)$, the function $s \mapsto K(\gamma(s))_{i_1\dots i_d} (\dot{\gamma}(s))^{i_1} \dots (\dot{\gamma}(s))^{i_d}$ is constant. The third equivalent definition is as follows. A contravariant symmetric tensor field $K^\sharp$ of rank $d$ defines a function $\cK$ on the cotangent bundle $T^*M$ which is a homogeneous polynomial of degree $d$ in the momenta $p_i$. For the contravariant metric tensor $g^{ij}$, this construction gives twice the Hamiltonian $\frac12 \sum_{ij} g^{ij}p_ip_j$. A contravariant symmetric tensor field $K^\sharp$ is Killing if and only if the corresponding function $\cK$ on $T^*M$ \emph{Poisson commute} with the Hamiltonian.

Contravariant Killing tensor fields of rank $d=1$ are called \emph{Killing vector fields}; it is well known that they correspond to infinitesimal isometries of the space $(M,ds^2)$, that is, are in one-to-one correspondence with the Lie algebra of the isometry group of $(M,ds^2)$. There is no such nice geometric description for Killing tensor fields of rank $d \ge 2$.

The space of all contravariant Killing tensor fields on $(M,ds^2)$ forms an associative, commutative, graded algebra $\sK(M)$ with respect to the symmetric tensor product. It contains a subalgebra generated by Killing vector fields. The elements of that subalgebra are called \emph{decomposable}. As the metric tensor $g_{ij}$ itself is trivially Killing, a natural question to ask is \emph{whether the whole algebra $\sK(M)$ is generated by Killing vector fields and the metric tensor}. For a general Riemannian manifold the answer is in the negative; an example is given by a Liouville metric $ds^2=(\la(x)+\mu(y)) (dx^2 + dy^2)$ and has been already discovered by Darboux.

On the other hand, any Killing tensor field on a space of constant curvature is decomposable~\cite{Tho,ST1,Tak}, and the structure of the algebra of Killing tensor fields is completely understood (note that for Riemannian symmetric spaces, we do not need to single out the metric tensor, as it is also decomposable). This result and the subsequent proof of the fact that Killing tensor fields on $\bc P^n$ are decomposable~\cite[Corollary~5]{E}, \cite[Theorem~2.2]{ST2}, led to the following modification of the original question: \cite[Question~3.9]{BMMT}: ``\emph{Is every Killing tensor field on a symmetric space decomposable?}'' There are two reasons making the class of symmetric spaces especially nice and natural in the context of the study of Killing tensors. The first one is the fact that symmetric spaces have a large isometry group, and hence a rich algebra of decomposable Killing tensor fields. The second, less obvious, reason is that for general Riemannian spaces, the integrability conditions for the overdetermined system~\eqref{eq:defK} are expressed in terms of the curvature tensor of $(M,ds^2)$ and its covariant derivatives; for symmetric spaces, the latter are all zeros.

Throughout the paper, we will almost exclusively work with covariant tensor fields; clearly, the algebras of covariant and of contravariant Killing tensor fields on a Riemannian manifold are isomorphic.

At present, we know that
\begin{itemize}
  \item
  all Killing tensor fields on the spaces of constant curvature~\cite{Tho,ST1,Tak},

  \item
  all Killing tensor fields on the complex projective spaces~\cite[Corollary~5]{E}, \cite[Theorem~2.2]{ST2},

  \item
  quadratic Killing tensor fields on some classical symmetric spaces: the classical groups~\cite{MNN}, the real Grassmannians $\SO(p+q)/\SO(p) \times \SO(q)$, the spaces $\SU(n)/\SO(n)$~\cite{NN} and the space $\SU(6)/\Sp(3)$~\cite{EL}, and 

  \item
  quadratic Killing tensor fields on some exceptional symmetric spaces: the group $\mathrm{G}_2$ and the space $\mathrm{G}_2/\SO(4)$~\cite{NN}
\end{itemize}
are decomposable.

However, there exist indecomposable quadratic Killing tensor fields on the following symmetric spaces:
\begin{itemize}
  \item
  the Cayley projective plane $\bo P^2 = \Ff/\Spin(9)$ \cite{MN1}.

  \item
  the quaternionic projective spaces $\bh P^n=\Sp(n+1)/(\Sp(n)\Sp(1))$ with $n \ge 3$ \cite{MN1}.

  \item
  the space $\mathrm{E}_6/\mathrm{F}_4$~\cite{EL}. 
\end{itemize}

Indecomposable quadratic Killing tensor fields on the quaternionic projective spaces $\bh P^n, \, n \ge 3$, are constructed explicitly (see formula~\eqref{eq:hTS} in Section~\ref{s:HPn}) and are the (nonzero) elements of a certain irreducible $\Sp(n+1)$-module $\sQ$ of dimension $\frac16(n - 2)(n + 1)(2n + 1)(2n + 3)$. By~\cite[Theorem~4]{MN3}, any \emph{quadratic} Killing tensor field on $\bh P^n, \, n \ge 3$, is the sum of an element of $\sQ$ and of a decomposable one.

Our main result states that this holds for the whole algebra of Killing tensor fields:
\begin{theorem} \label{t:hpn}
The algebra of (covariant) Killing tensor fields on the quaternionic projective space $\bh P^n, \, n \ge 3$, is generated by Killing covector fields and the quadratic Killing tensor fields from $\sQ$. All Killing tensor fields on $\bh P^n, \, n \le 2$, are decomposable.
\end{theorem}

Our method also works for $\bc P^n$ giving yet another proof of the following result:

\begin{theorem}[{\cite{E}, \cite[Theorem~2.2]{ST2}}] \label{t:cpn}
Any Killing tensor field on the complex projective space $\bc P^n, \, n \ge 1$, is decomposable.
\end{theorem}

We will first give the proof of Theorem~\ref{t:cpn} (Section~\ref{s:CPn}): it is much easier and shows the ideas behind our approach. Then in Section~\ref{s:HPn} we establish Theorem~\ref{t:hpn}.

Notice that by~\cite[Theorem~2]{MN3}, the results similar to those in Theorem~\ref{t:hpn} and~\ref{t:cpn} hold for the noncompact duals: for the Killing tensor fields on the complex and quaternionic hyperbolic spaces.

The authors are thankful to Alexey Bolsinov, Holger Dullin and Vladimir Matveev for useful discussions.

\section{Preliminaries}
\label{s:prel}

\subsection{Killing tensor fields on the sphere}
\label{ss:sphere}

We view the unit sphere $S^m, \, m \ge 1$, as the round sphere in $\br^{m+1}$. Any covariant Killing tensor field $K$ of rank $d \ge 0$ is a homogeneous polynomial of degree $d$ in the Killing covector fields on $S^m$, relative to the symmetric tensor product by~\cite{ST1, Tak, Tho}.

We identify the tangent bundle to the space $\br^{m+1}$ with the space $\br^{2m+2} = \br^{m+1} \oplus \br^{m+1}$ of pairs of vectors $(X,P), \, X, P \in \br^{m+1}$, and introduce orthonormal bases $\{e_i\}$ and $\{f_i\},\, i=1, \dots, m+1$, for the two copies of the space $\br^{m+1}$, respectively, in such a way that the element $\sum_{i=1}^{m+1} x_i e_i + \sum_{i=1}^{m+1} p_i f_i \in \br^{2m+2}$ corresponds to the element $\sum_{i=1}^{m+1} x_i e_i + \sum_{i=1}^{m+1} p_i \partial/\partial{x_i} \in T \br^{m+1}$. Then any Killing covector field $K$ on $S^m$ is defined as follows: $K(X)(P)=\<AX, P\>$, for $X \in S^m \subset \br^{m+1}$ and $P \in T_XS^m$, where $A \in \so(n+1)$. This gives a linear isomorphism between the space of Killing covector fields on $S^m$ and the space spanned by the polynomials $V_{jk}=V_{jk}(X,P), \, j,k=1, \dots, m+1$, defined by $V_{jk}(X)(P) = x_jp_k-x_kp_j$, for $X, P \in \br^{m+1}$, where $x_j=\<X,e_j\>$ and $p_j=\<P,f_j\>$ are the coordinates of $X$ and $P$ relative to the corresponding bases, respectively. This defines a one-to-one correspondence between Killing covariant tensor fields of rank $d$ on $S^m$ and homogeneous polynomial functions of degree $d$ in the polynomial functions $V_{jk}(X,P)$ (note however, that when $m \ge 3$, the polynomials $V_{jk}(X,P)$ are not algebraically independent -- they satisfy Pl\"{u}cker identities, and so the homomorphism $v_{jk} \mapsto V_{jk}(X,P)$ from the polynomial algebra $\br[v_{jk}] = \br[\so(m+1)]$ to the algebra of polynomial functions on $\br^{2m+2}$ generated by the polynomials $V_{jk}(X,P)$ has non-trivial kernel~\cite{MMS}). For a Killing covariant tensor field $K$ on $S^m$, we define $\Phi_K=\Phi_K(X,P)$ the corresponding polynomial function on $\br^{2m+2}$. We note that for $K$ of rank $d$, the polynomial $\Phi_K(X,P)$ is bi-homogeneous of degree $d$ in the components of $X$ and of degree $d$ in the components of $P$. Moreover, any $V_{jk}$, and hence any $\Phi_K$, is invariant relative to the action of the group $\SL(2)$ on $\br^{2m+2}$ defined by
  \begin{equation}\label{eq:SL2action}
  A.(X,P) = (a_{11}X + a_{21}P, a_{12}X + a_{22}P), \text{ for } A=\begin{pmatrix} a_{11} & a_{12} \\ a_{21} & a_{22} \end{pmatrix} \in \SL(2).
  \end{equation}
The converse is also true: any polynomial on $\br^{2m+2}=\{(X,P) \, | \, X, P \in \br^{n+1}\}$ which is invariant under the $\SL(2)$ action given in~\eqref{eq:SL2action} is a polynomial in the polynomials $V_{jk}$ (by the First Fundamental Theorem of Invariant Theory for $\SL(2)$~\cite[Section~11.1.2]{P}). We can therefore say that to any Killing covariant tensor field $K$ on $S^m$ of rank $d$ there corresponds a unique polynomial $\Phi_K(X,P)$ homogeneous of degree $d$ in both $X$ and $P$ and invariant under the $\SL(2)$ action given by~\eqref{eq:SL2action}.

\subsection{Killing tensor fields under a Riemannian submersion}
\label{ss:subm}

Consider a Riemannian submersion $\pi: P \to B$ of a principal $G$-bundle $P$ to its space of orbits $B$ of a connected Lie group $G$. Any rank $d$ $G$-invariant symmetric covariant tensor $K$ on $P$ descends to a canonical invariant symmetric covariant tensor $\overline{K}$ on $B$ as follows: for $b \in B$ select any $p \in \pi^{-1}(\{b\})$ in the fibre over it and define $\overline{K}_b(Y^d) =K_p(X^d)$ where $X \in \mathcal{H}_p \subset T_p P$ is the unique horizontal lift of the tangent vector $Y  \in T_b B$ (i.e., $\pi_*(X)=Y$), and where $Y^d$ (respectively, $X^d$) denotes the diagonal $d$-tuple of the elements $Y$ (respectively, $X$).

To see this is well defined, given a second point $p^\prime \in \pi^{-1}(\{b\})$ let $\xi = v^*$ be the fundamental field of $v \in \g$ of the one parameter group that links $p$ to $p^\prime$. Extend $Y$ to a vector field and consider the unique basic vector field $X$ such that $\pi_*(X)=Y$ (the so-called horizontal lift of $Y$). Since $K$ is $G$-invariant $\mathcal{L}_\xi K = 0$. Thus $\xi K(X^d)-dK(X^{d-1},[\xi,X]) = 0$ and $[\xi,X]=0$ since $\xi$ is fundamental and $X$ is basic. Thus $\xi K(X^d) = 0$ and thereby $K(X^d)$ is constant along the orbit of the one parameter group and $K_p(X^d) = K_{p^\prime}(X^d)$.

 \begin{lemma}\label{l: descent}
 	Given a rank $d$ $G$-invariant symmetric covariant Killing tensor $K$ on a principal $G$-bundle $\pi:P \to B$, the rank $d$ symmetric covariant tensor $\overline{K}$ defined above is also Killing.
 \end{lemma}

 \begin{proof}
 	Given a vector field on $Y$ on $B$ let $X$ be its unique horizontal lift.  Then \begin{align*}
 		\left(\overline{\nabla}_Y \overline{K}\right)(Y^d) = Y\overline{K}(Y^d) - d\overline{K}(Y^{d-1},\overline{\nabla}_Y Y) & = XK(X^d)-dK(X^{d-1},(\nabla_X X)^\mathcal{H})\\
 		& = XK(X^d)-dK(X^{d-1},\nabla_X X) = 0
 \end{align*}
  since $A_X X=0$ for the O'Neill tensor $A$.
 \end{proof}

 It is natural to ask if for a principal $G$-bundle $\pi:P \to B$, any rank $d$ covariant Killing tensor $L$ on $B$ descends from a rank $d$ $G$-invariant Killing tensor $K$ on $P$ (i.e., $\overline{K}=L$)? The answer in general will be in the negative; but, with considerable effort, in the case of the principal $G$-bundles given by the Hopf fibrations $\pi: S^{2n+1} \to \bc P^n$ and $\pi: S^{4n+3} \to \bh P^n$, where $G=S^1$ and $G=\Sp(1)$ respectively, we show in the proofs of Proposition~\ref{p:liftC} and Proposition~\ref{p:liftH} respectively, that the answer is in the affirmative.

\section{Killing tensors on \texorpdfstring{$\bc P^n$}{\unichar{"2102}P\textnsuperior}. Proof of Theorem~\ref{t:cpn}}
\label{s:CPn}

Denote $J$ the complex structure on $\br^{2n+2},\; n \ge 1$, and denote $\xi = J N$ the Reeb vector fields on the unit sphere $S^{2n+1} \subset \br^{2n+2}$, where $N$ is the unit outward normal to $S^{2n+1}$. The group $S^1 = \{e^tJ \, | \, t \in \br\} \subset \SO(2n+2)$ acts on $S^{2n+1}$ by isometries, with equidistant geodesic orbits isometric to the unit circle. This defines the Hopf fibration $\pi: S^{2n+1} \to \bc P^n = S^{2n+1}/S^1$, with the projection map $\pi$ being a Riemannian submersions with geodesic fibers. The action of $S^1$ preserves both the horizontal distribution $\cH$ and the vertical distribution $\cV$. Any $S^1$-invariant (continuous) vertical vector field is a constant multiple of $\xi$. 

\begin{proof}[Proof of Theorem~\ref{t:cpn}] The claim of the theorem follows from the two propositions below.
%

  \begin{proposition} \label{p:liftC}
    For any covariant Killing tensor field $L$ on $\bc P^n$, there exists an $S^1$-invariant, covariant Killing tensor field $K$ on $S^{2n+1}$ such that $\overline{K}=L$.
  \end{proposition}

  The proof of Proposition~\ref{p:liftC} is given in Subsection~\ref{ss:liftCPn} below. It is constructive: given a Killing tensor field $L$ 
   on $\bc P^n$, we first pull it back to 
   $\pi^*L$ on $S^{2n+1}$. The tensor field 
   $\pi^*L$ is not in general Killing on $S^{2n+1}$. We give explicit formulas for the terms 
   that have to be added to 
   $\pi^*L$ to make the resulting tensor Killing.

  Furthermore, under the one-to-one identification $K \mapsto \Phi_K$ given in Subsection~\ref{ss:sphere}, the space of Killing, $S^1$-invariant tensor fields $K$ on $S^{2n+1}$ is linearly isomorphic to the space of polynomials in $\br^{4n+4} = \{(X,P) \, | \, X, P \in \br^{2n+2}\}$, which are invariant under the action of the group $\SL(2)$ given in~\eqref{eq:SL2action} and are invariant under the action of the group $S^1$ given by $t.(X,P) = (e^{tJ}X, e^{tJ}P)$, for $t \in \br$ (and are homogeneous in $X$ and in $P$, of the same degree).

  One family of such polynomials is given by the space $\cL$ of polynomials coming from $S^1$-invariant Killing covector fields: any such covector field corresponds under $\Phi$ to the polynomial
  \begin{equation}\label{eq:cvf}
  T_A= \<AX, P\>, \quad \text{where } A \in \ug(n+1)
  \end{equation}
  (here $\ug(n+1)$ is the subalgebra of $\so(2n+2)$ which commutes with $J$).

  The proof of Theorem~\ref{t:cpn} is now completed by the following proposition:
  \begin{proposition} \label{p:polyC}
    The algebra of polynomials on $\br^{4n+4} = \{(X,P) \, | \, X, P \in \br^{2n+2}\}$ which are invariant under the simultaneous actions of the group $\SL(2)$ given in~\eqref{eq:SL2action} and the group $S^1$ given by $t.(X,P) = (e^{tJ}X, e^{tJ}P)$ is generated by the polynomials $T_A \in \cL$ given by~\eqref{eq:cvf}.
  \end{proposition}
  We prove Proposition~\ref{p:polyC} in Subsection~\ref{ss:decoCPn}.
\end{proof}

\subsection{Lifting Killing tensors from \texorpdfstring{$\bc P^n$}{\unichar{"2102}P\textnsuperior} to \texorpdfstring{$S^{2n+1}$}{S\texttwosuperior\textnsuperior\textplussuperior\textonesuperior}. Proof of Proposition~\ref{p:liftC}}
\label{ss:liftCPn}


From O'Neill formulas we know that for basic vector fields $Y$ and $Z$, $(\nabla_Y Z)^\mathcal{H}$ is also basic and the following holds:
\begin{equation}\label{eq:nablaC}
  \nabla_Y Z = (\nabla_Y Z)^\mathcal{H} -\<J Y,Z\> \xi, \qquad  \qquad \nabla_\xi \xi = 0, \qquad \nabla_Y \xi = \nabla_\xi Y = J Y,
\end{equation}
where $\nabla$ and $\on$ are the Levi-Civita connections on $S^{2n+1}$ and $\bc P^n$, respectively and $\on_{\pi_* Y} \pi_* Z = \pi_*(\nabla_Y Z)$.

For $d \ge 0$, denote $\ocS^d$ the space of smooth, symmetric, covariant tensor fields of rank $d$ on $\bc P^n$, and denote $\cS^d$ the space of smooth, symmetric, $S^1$-invariant, covariant tensor fields of rank $d$ on $S^{2n+1}$.

Let $\omega$ be the $1$-form dual to the vector field $\xi$. A tensor field $K \in \cS^d$ has the decomposition
\begin{equation}\label{eq:Jinv}
  K = \pi^*L_d + \pi^*L_{d-1} \odot \omega + \pi^*L_{d-2} \odot \omega^{\odot 2} + \dots + \pi^*L_1 \odot \omega^{\odot (d-1)} + \pi^*L_0 \omega^{\odot d},
\end{equation}
where the tensor fields $L_m \in \ocS^m,\, m = 0, \dots, d$. 
Note that $\overline{K} = L_d$.

Define the following operators:
\begin{equation}\label{eq:def4opC}
\begin{alignedat}{3}
  \od: \, & \ocS^m \to \ocS^{m+1} & &\qquad\qquad& (\od T)(Y^{m+1})& = (m+1) (\on_Y T)(Y^m),\\
  \oi: \, & \ocS^m \to \ocS^{m-1} & &\qquad\qquad& (\oi T)(Y^{m-1})& = \sum\nolimits_{i=1}^{2n} (\on_{e_i} T)(Y^{m-1},J e_i), \\
  \orJ: \, & \ocS^m \to \ocS^m & &\qquad\qquad& (\orJ T)(Y^m)& = m \, T (Y^{m-1},J Y), \\
  \ot: \, & \ocS^m \to \ocS^{m-2} & &\qquad\qquad& (\ot T)(Y^{m-2})& = \sum\nolimits_{i=1}^{2n} T (Y^{m-2},e_i, e_i), \\
\end{alignedat}
\end{equation}
where $Y$ is a vector field on $\bc P^n$, $\{e_i\}$ is a (local) orthonormal frame on $\bc P^n$, and where we denote $Y^k$ the $k$-tuple consisting of $k$ copies of $Y$. We define $\ocS^m = 0$ when $m < 0$.

In terms of the tensor fields 
$L_m$, the Killing condition~\eqref{eq:defK} for $K$ is given as follows.
\begin{lemma} \label{l:deltaJ}
  A tensor field $K \in \cS^d$ given by~\eqref{eq:Jinv} is Killing on $S^{2n+1}$ if and only if on $\bc P^n$, the following equations are satisfied:
  \begin{equation}\label{eq:deltaJ}
     \od \; L_d = 0 \qquad \text{and} \qquad \od \; L_{m-1} = 2m \orJ \; L_m, \quad m = d, \dots, 1.
  \end{equation}
\end{lemma}
\begin{proof}
  Let $X$ be a basic vector field on $S^{2n+1}$ and denote $Y=\pi_*X$ its projection to $\bc P^n$. Given a point $q \in S^{2n+1}$, we can assume that $\on_Y Y = 0$ at the point $p = \pi(q) \in \bc P^n$. Then at the point $q$, the O'Neill formulas~\eqref{eq:nablaC} give $\nabla_X \xi = \nabla_\xi X = JX$ and $\nabla_X X = \nabla_\xi \xi = 0$. The Killing condition for $K$ means that $\nabla_{X+\xi} K((X+\xi)^d) = 0$. As both $K$ and $X+\xi$ are $S^1$-invariant, we have $\xi(K((X+\xi)^d)) = 0$ which gives $X (K((X+\xi)^d)) = 2d K((X+\xi)^{d-1},JX)$. Now from~\eqref{eq:Jinv} we have $K(X+\xi) = \sum_{m=0}^{d} \pi^*L_m(X^m)$, and so the Killing condition for $K$ is equivalent to the system of equations $(\on_Y L_d)(Y^d)=(\nabla_X \pi^*L_d)(X^d) = 0$ and $(\on_Y L_{m-1})(Y^m) = (\nabla_X \pi^*L_{m-1})(X^m) = 2m L_m(Y^{m-1},JY)$, for $m = d, \dots, 1$. 
  This system of equations is equivalent to~\eqref{eq:deltaJ}, by using the definitions of $\od$ and $\orJ$ given in~\eqref{eq:def4opC}.
\end{proof}

The first equation of~\eqref{eq:deltaJ} tells that the tensor field $L_d \in \ocS^d$ is Killing on $\bc P^n$. Hence to prove Proposition~\ref{p:liftC} we need to construct, starting with an arbitrary Killing tensor field $L_d \in \ocS^d$, tensor fields $L_{m-1} \in \ocS^{m-1}, \, m = d, \dots, 1$, such that the remaining equations of~\eqref{eq:deltaJ} are satisfied.

This is done with the help of the following commutation relations.
\begin{lemma} \label{l:brackets}
  In the notation of~\eqref{eq:def4opC}, the following holds:
  \begin{equation}\label{eq:brackets}
    [\od, \oi] = -4(n+m) \orJ, \qquad  [\od, \ot] = 2 [\oi, \orJ], \qquad [\ot, \orJ] = 0,
  \end{equation}
  where in the first equation, both sides act on the elements from $\ocS^m$.
\end{lemma}
\begin{proof}
  The third equation of~\eqref{eq:brackets} follows from the definitions of $\od$ and $\orJ$ and the fact that $\sum_{i=1}^{2n} T(Y^k,e_i,Je_i) = 0$, for any $T \in \ocS^{k+2}$, as $T$ is symmetric and $J$ is skew-symmetric.

  To simplify calculations in the first two equations, we can assume that a vector field $Y$ and an orthonormal frame $\{e_i\}$ are chosen in such a way that $\on Y = \on e_i = 0$ at a point $p \in \bc P^n$. For the second equation of~\eqref{eq:brackets}, we compute $([\oi, \orJ] T)(Y^{m-1}) = -\sum_{i=1}^{2n} (\on_{e_i} T)(Y^{m-1},e_i)$ and $([\od, \ot] T)(Y^{m-1}) = 2\sum_{i=1}^{2n} (\on_{e_i} T)(Y^{m-1},e_i)$, for $T \in \ocS^m$, and the claim follows.

  The proof of the first equation is a little bit more involved. For $T \in \ocS^m$, we obtain $\od \, \oi \, T, \, \oi \, \od \, T \in \ocS^m$. At the point $p$ where $\on Y = \on e_i = 0$ (and as $\on J = 0$) we find
  \begin{align*}
    (\od(\oi T))(Y^m) & = m (\on_Y (\oi T))(Y^{m-1}) = m Y \big((\oi T)(Y^{m-1})\big) \\
    &= m \sum\nolimits_{i=1}^{2n} Y \big((\on_{e_i} T)(Y^{m-1},J e_i)\big) = m \sum\nolimits_{i=1}^{2n} (\on^2_{Y,e_i} T) (Y^{m-1},J e_i).
  \end{align*}
  and
  \begin{align*}
    (\oi(\od T))(Y^m) & = \sum\nolimits_{i=1}^{2n} (\on_{e_i} \od T) (Y^{m},J e_i) = \sum\nolimits_{i=1}^{2n} e_i \big((\od T) (Y^{m},J e_i) \big) \\
     & = \sum\nolimits_{i=1}^{2n} e_i \big(m(\on_Y T) (Y^{m-1},J e_i) + (\on_{J e_i} T) (Y^m)\big) \\
     & = \sum\nolimits_{i=1}^{2n} \big(m(\on^2_{e_i,Y} T) (Y^{m-1},J e_i) + (\on^2_{e_i,J e_i} T) (Y^m)\big),
  \end{align*}
  from which
  \begin{equation*}
    ([\od,\oi] T)(Y^m) = \sum\nolimits_{i=1}^{2n} \big(m(\overline{R}(Y,e_i) T) (Y^{m-1},J e_i) - \tfrac12 (\overline{R}(e_i,J e_i) T) (Y^m)\big),
  \end{equation*}
  where in the second term in the summation we used the fact that $\sum\nolimits_{i=1}^{2n} (\on^2_{e_i,J e_i} T) (Y^m) = - \sum\nolimits_{i=1}^{2n} (\on^2_{J e_i,e_i} T) (Y^m)$, and where $\overline{R}$ is the curvature tensor of $\bc P^n$, which is given by $\overline{R}(X,Z)W = \<Z, W\> X - \<X, W\> Z - 2 \<JX, Z\> JW + \<JW, X\> JZ - \<JZ, W\> JX$. Then $\sum\nolimits_{i=1}^{2n} (\on^2_{e_i,Y} T) (Y^{m-1},J e_i) = 2(1-2m-n) T (Y^{m-1},JY)$ and $\tfrac12\sum\nolimits_{i=1}^{2n} (\on^2_{e_i,J e_i} T) (Y^m) = 2m(n+1) T (Y^{m-1},JY)$, and the claim follows.
\end{proof}

Now we claim that for a given Killing tensor field $L_d \in \ocS^d$, the tensor fields $L_{d-1}, L_{d-2},$ $\dots, L_1, L_0$ consecutively defined by
\begin{equation}\label{eq:defLm}
\begin{gathered}
  L_{d-1} = C_d \, \oi \, L_d, \qquad L_{m-1} = C_m (\oi \, L_{m} - (m+1) \ot \, L_{m+1}), \quad m = d-1, d-2, \dots, 1, \\
  \text{where} \; C_m = -m((2n+d+m)(d-m+1))^{-1},
\end{gathered}
\end{equation}
satisfy equations~\eqref{eq:deltaJ}.

Indeed, according to~\eqref{eq:defLm} we have $\od \, L_{d-1} = C_d \, \od \, \oi L_d = C_d \, ([\od, \oi] + \oi \, \od) L_d = -4 C_d (n+d) \orJ \, L_d = 2d \, \orJ \, L_d$, as required by~\eqref{eq:deltaJ}, where in the second last equation we used the formula for $[\od, \oi]$ given in~\eqref{eq:brackets} and the fact that $\od \, L_d = 0$, and in the last one, the expression for $C_d$ given in~\eqref{eq:defLm}.

Furthermore, having already constructed the tensor fields $L_{d-1}, \dots, L_{m}$ for some $m \le d-1$, we obtain from~\eqref{eq:defLm}:
\begin{equation*} 
\begin{split}
    \od \, L_{m-1} & = C_m (([\od, \oi] + \oi \, \od) L_{m} - (m+1) ([\od, \ot] + \ot \, \od) L_{m+1}) \\
    &
    {\begin{aligned}
    = C_m (-4(n+m) \orJ \, L_{m} & + 2(m+1) \oi \, \orJ \, L_{m+1} \\
    & - (m+1) (2 [\oi, \orJ] L_{m+1} + 2(m+2) \ot \, \orJ \, L_{m+2}))
    \end{aligned}} \\
    & = C_m \orJ \, (-4(n+m) L_{m} + 2(m+1) (\oi \, L_{m+1} - (m+2) \ot \, L_{m+2})) \\
    & = C_m \orJ \, (-4(n+m) L_{m} + 2(m+1) C_{m+1}^{-1} L_{m}) = 2m \orJ \, L_{m},
\end{split}
\end{equation*}
as required, where in the second equation we used the formulas for $[\od, \oi]$ and $[\od, \ot]$ from~\eqref{eq:brackets} and the expressions for $\od \, L_m$ and $\od \, L_{m+1}$ from~\eqref{eq:deltaJ},
in the third equation, the fact that $[\ot, \orJ] = 0$ from~\eqref{eq:brackets}, in the fourth, the expression for $L_{m}$ from~\eqref{eq:defLm}, and in the fifth, the formulas for $C_m$ and $C_{m+1}$ given in~\eqref{eq:defLm}.

Thus given a covariant Killing tensor field $L_d$ on $\bc P^n$, we construct the tensor fields $L_{d-1}, L_{d-2}, \dots, L_1, L_0$ on $\bc P^n$ by~\eqref{eq:defLm}, 
pull them back to tensor fields $\pi^*L_m$ on $S^{2n+1}$ and then compose a symmetric, covariant tensor field $K$ on $S^{2n+1}$ by~\eqref{eq:Jinv}. The tensor field $K$ is Killing and $S^1$-invariant, and $\overline{K}=L_d$, as required. This completes the proof of Proposition~\ref{p:liftC}.

\subsection{Proof of Proposition~\ref{p:polyC}}
\label{ss:decoCPn}

In the space $\br^{4n+4} = \br^{2n+2} \oplus \br^{2n+2} = \{(X, P) \, | \, X, P \in \br^{2n+2}\}$, we introduce orthonormal bases $\{e_i\}$ and $\{f_i\},\, i=1, \dots, 2n+2$, for the two copies of the space $\br^{2n+2}$ as in Subsection~\ref{ss:sphere}, in such a way that $Je_{2j-1}=e_{2j}, \, Je_{2j}=-e_{2j-1}$, $Jf_{2j-1}=f_{2j}, \, Jf_{2j}=-f_{2j-1}$, for $j = 1, \dots, n+1$. Relative to these bases, the space $\cL$ of polynomials $T_A$ given by~\eqref{eq:cvf} is spanned by the polynomials
\begin{equation}\label{eq:KvfC}
\begin{split}
   W_{jk,0} & = x_{2j-1}p_{2k-1}-x_{2k-1}p_{2j-1}+x_{2j}p_{2k}-x_{2k}p_{2j} \quad \text{and}\\
   W_{jk,1} & = x_{2j-1}p_{2k}-x_{2k}p_{2j-1}-x_{2j}p_{2k-1}+x_{2k-1}p_{2j},
\end{split}
\end{equation}
for $1 \le j, k \le n+1$.

We complexify the space $\br^{4n+4}$ and consider the action of the group $\SL(2,\bc) \times \bc^*$ on the space $\bc^{4n+4}$ defined by the same formula~\eqref{eq:SL2action}, but with complex coefficients, and by $z.(X,P) = (e^{zJ}X, e^{zJ}P)$, where $z \in \bc$ (or by $z.(X,P) = (e^{(\ln z) J}X, e^{(\ln z) J}P)$ for $z \in \bc^*$), respectively. Note that these two actions commute and that $(\SL(2,\bc) \times \bc^*)/\mathbb{Z}_2 \cong \GL(2,\bc)$, where $\mathbb{Z}_2 = \{(I_2,1),(-I_2,-1)\}$ is the ineffective kernel of the action.

Introduce a new basis $\{E_i,F_i\},\, i = 1, \dots, 2n+2$, for $\bc^{4n+4}$ by defining $E_{2j-1}=e_{2j-1} + \ri e_{2j}, \, E_{2j}=e_{2j-1} - \ri e_{2j}, \, F_{2j-1}=f_{2j-1} + \ri f_{2j}, \, F_{2j}=f_{2j-1} - \ri f_{2j}$, for $j=1, \dots, n+1$. The corresponding coordinate changes are given by $z_{2j-1}=\frac12(x_{2j-1} - \ri x_{2j}), \, z_{2j}=\frac12(x_{2j-1} + \ri x_{2j}), \, w_{2j-1}=\frac12(p_{2j-1} - \ri p_{2j}), \, w_{2j}=\frac12(p_{2j-1} + \ri p_{2j})$, where $\sum_{k=1}^{2n+2} (x_ke_k+p_kf_k) = \sum_{k=1}^{2n+2} (z_kE_k+w_kF_k)$. Then the subspaces $\Span_\bc(E_{2j-1}, F_{2j-1})$ and $\Span_\bc(F_{2j}, -E_{2j})$ are $\GL(2,\bc)$-invariant. Moreover, the action of an element $M \in \GL(2,\bc)$ on each of these subspaces, relative to these particular bases, is given as follows:
\begin{equation*}
  M.u^{(j)} = M u^{(j)}, \quad M.v^{(j)} = (M^t)^{-1} v^{(j)}, \quad \text{where } u^{(j)}=\begin{pmatrix} z_{2j-1} \\ w_{2j-1} \end{pmatrix}, \quad v^{(j)}=\begin{pmatrix} -w_{2j} \\ z_{2j} \end{pmatrix},
\end{equation*}
and so $\GL(2, \bc)$ acts by its standard representation on each of the $n+1$ vectors $u^{(j)}$ and by the dual representation on each of the $n+1$ vectors $v^{(j)}$. By the First Fundamental Theorem of Invariant Theory~\cite[Theorem~9.1.4]{P}, the algebra of invariant polynomials of this action is generated by the polynomials $(u^{(j)})^t v^{(k)}, \; j,k=1, \dots, n+1$. From~\eqref{eq:KvfC} we obtain
\begin{equation*}
(u^{(j)})^t v^{(k)} = w_{2j-1} z_{2k}-z_{2j-1}w_{2k} =  - \tfrac14 (W_{jk,0} + \ri W_{jk,1}).
\end{equation*}
Note that $W_{kj,0}=-W_{jk,0}$ and $W_{kj,1}=W_{jk,1}$, and so the algebra of complex polynomial on $\bc^{4n+4}$ invariant under our action of the group $\GL(2,\bc)$ is generated by the polynomials $W_{jk,0}$ and $W_{jk,1}$ from the complexification of the space $\cL$. Now a real polynomial invariant under the action of the group $\SL(2) \times S^1$ will be also invariant under the action of the group $\GL(2,\bc)$ and so will be a complex polynomial in $W_{jk,0}$ and $W_{jk,1}$, and hence a real polynomial in $W_{jk,0}$ and $W_{jk,1}$, as the polynomials $W_{jk,0}$ and $W_{jk,1}$ are real.

This completes the proof of Proposition~\ref{p:polyC}.

\section{Killing tensors on \texorpdfstring{$\bh P^n$}{\unichar{"210D}P\textnsuperior}. Proof of Theorem~\ref{t:hpn}}
\label{s:HPn}

Denote $J_1,J_2,J_3$ complex structures on $\br^{4n+4},\; n \ge 1$, such that $J_1J_2=J_3$, and denote $\xi_\alpha = J_\alpha N,\, \alpha=1,2,3$, the corresponding orthonormal vector fields on the unit sphere $S^{4n+3} \subset \br^{4n+4}$, where $N$ is the unit outward normal to $S^{4n+3}$. The group $\Sp(1) \subset \SO(4n+4)$ with the Lie algebra $\spg(1) = \Span(J_1,J_2,J_3)$ acts on $S^{4n+3}$ by isometries, with totally geodesic orbits isometric to the unit sphere $S^3$. This defines the Hopf fibration $\pi: S^{4n+3} \to \bh P^n$, with the projection map $\pi$ being a Riemannian submersions with totally geodesic fibers. We denote $\cH$ and $\cV$ the horizontal and the vertical distributions, respectively. Note that the distribution $\cH$ is not integrable. The action of $\Sp(1)$ preserves the distributions $\cH$ and $\cV$. An $\Sp(1)$-invariant horizontal vector field is called \emph{basic}. Similarly, a horizontal tensor field on the sphere $S^{4n+3}$ (a section of the tensor bundle over $\cH$) is called \emph{basic} if it is $\Sp(1)$-invariant.

At every point of the sphere $S^{4n+3}$, the vector fields $\xi_1, \xi_2, \xi_3$ form an orthonormal vertical basis. Note that these vector fields are Killing, but not $\Sp(1)$-invariant. For a vertical vector field $\xi$ and a horizontal vector field $Y$ we denote $J_\xi Y = \sum_{\alpha=1}^{3} \<\xi,\xi_\alpha\> J_\alpha Y$.

\begin{proof}[Proof of Theorem~\ref{t:hpn}] The claim of the theorem follows from the two propositions below.
%

  \begin{proposition} \label{p:liftH}
    For any Killing tensor field $L$ on $\bh P^n$, there exists a Killing, $\Sp(1)$-invariant tensor field $K$ on $S^{4n+3}$ such that $\overline{K}=L$.
  \end{proposition}

  The proof of Proposition~\ref{p:liftH} is given in Subsection~\ref{ss:liftHPn} below, with some technicalities moved to Subsection~\ref{ss:tensorial}. It is constructive: given a rank $d$ covariant Killing tensor field $L$ on $\bh P^n$, we pull it back to a covariant tensor field $L_d=\pi^*L$ on $S^{4n+3}$. The tensor field $L_d$ is $\Sp(1)$-invariant, but need not to be Killing on $S^{4n+3}$. Then we consecutively add
  terms to $L_d$ to make the resulting tensor Killing.

  Now we consider Killing, $\Sp(1)$-invariant tensor fields $K$ on $S^{4n+3}$. Under the map $K \mapsto \Phi_K$ given in Subsection~\ref{ss:sphere}, they are in one-to-one correspondence with the polynomials on $\br^{8n+8} = \{(X,P) \, | \, X, P \in \br^{4n+4}\}$, which are invariant under the action of the group $\SL(2)$ given in~\eqref{eq:SL2action} and are invariant under the action of the group $\Sp(1)$ given by $U.(X,P) = (UX, UP)$, for $U\in \Sp(1)$ (and are homogeneous in $X$ and in $P$, of the same degree).

  We already know two families of such polynomials. The first one is the space $\cL$ of polynomials coming from the $\Sp(1)$-invariant Killing covector fields: any such covector field corresponds under $\Phi$ to the polynomial
  \begin{equation}\label{eq:hvf}
  T_A= \<AX, P\>, \quad \text{where } A \in \spg(n+1),
  \end{equation}
  where $\spg(n+1)$ is the subalgebra of $\so(4n+4)$ which commutes with $\spg(1)$. The second family $\cQ$ has been constructed in~\cite[Section~2]{MN1}. Denote $V_{n+1}$ the space of symmetric operators on $\br^{4n+4}$ which commute with $\spg(1)$. For $S_1, S_2 \in V_{n+1}$, we define the polynomial $T_{S_1,S_2}$ on $\br^{8n+8}$ by
  \begin{equation} \label{eq:hTS}
    T_{S_1,S_2}(X,P)=\sum\nolimits_{\a=1}^{3} \<S_1J_\a X, P\>\<S_2J_\a X, P\>,
  \end{equation}
  for $X,P \in \br^{4n+4}$. As it is shown in~\cite[Section~2]{MN1}, every polynomial $T_{S_1,S_2}$ is $\Sp(1)$-invariant and moreover, the quadratic, covariant tensor field $K=K_{S_1,S_2}$ on the sphere $S^{4n+3}$ such that $\Phi_{K_{S_1,S_2}} =T_{S_1,S_2}$ is Killing. When $n \ge 3$, the space of such quadratic Killing tensor fields contains a subspace $\sQ$, an irreducible $\Sp(n+1)$-module of dimension $\frac16(n - 2)(n + 1)(2n + 1)(2n + 3)$ all of whose nonzero elements are indecomposable~\cite[Theorem~1]{MN1}. We denote $\cQ=\Span(T_{S_1,S_2} \, | \, S_1,S_2 \in V_{n+1})$.

  We prove that these two families generate the whole algebra of invariant polynomials:
  \begin{proposition} \label{p:polyH}
    The algebra of polynomials on $\br^{8n+8} = \{(X,P) \, | \, X, P \in \br^{4n+4}\}$ which are invariant under the simultaneous actions of the group $\SL(2)$ given in~\eqref{eq:SL2action} and the group $\Sp(1)$ given by $U.(X,P) = (UX, UP)$ is generated by the polynomials $T_A \in \cL$ given by~\eqref{eq:hvf} and the polynomials $T_{S_1,S_2} \in \cQ$ given by~\eqref{eq:hTS}.
  \end{proposition}

  The proof of Proposition~\ref{p:polyH} is given in Subsection~\ref{ss:decoHPn}. It should be noted that the generators in Proposition~\ref{p:polyH} are far from being algebraically independent. First of all, one can show that the $\Sp(n+1)$-modules $\cQ$ and $\Sym^2(\cL)$ have some common nontrivial irreducible components. Moreover, one can check (or deduce from the proof) that the product of any two elements of $\cQ$ is a fourth degree polynomial in the elements of $\cL$; see Remark~\ref{rem:SFT} for more details.

  Modulo Propositions~\ref{p:liftH} and~\ref{p:polyH}, the proof of Theorem~\ref{t:hpn} is completed by~\cite[Theorem~4(a)]{MN3}. Indeed, it remains to show that any Killing tensor field on $\bh P^n, \, n \le 2$, is decomposable. Now, from Propositions~\ref{p:liftH} and~\ref{p:polyH} we deduce that the algebra of Killing tensor fields on $\bh P^n$, for all $n \ge 1$, is generated by linear and quadratic Killing fields, and then~\cite[Theorem~4(a)]{MN3} implies that for $n \le 2$, any quadratic Killing tensor field is decomposable.
\end{proof}

\subsection{Lifting Killing tensors from \texorpdfstring{$\bh P^n$}{\unichar{"210D}P\textnsuperior} to \texorpdfstring{$S^{4n+3}$}{S\textfoursuperior\textnsuperior\textplussuperior\textthreesuperior}. Proof of Proposition~\ref{p:liftH}}
\label{ss:liftHPn}

The idea of the proof is somewhat similar to that of Proposition~\ref{p:liftC} given in Subsection~\ref{ss:liftCPn}. However, as the vertical distribution $\cV$ is no longer $1$-dimensional, we will work on the sphere $S^{4n+3}$ rather than on the space $\bh P^n$, and the proof will be much more technically involved. We start with a rank $d$ Killing tensor field $L$ on $\bh P^n$ pullback to $L_d=\pi^*L$ on $S^{4n+3}$, an $\Sp(1)$-invariant section of the bundle $\Sym^d(\cH)$. The tensor field $L_d$ is not Killing on $S^{4n+3}$.
To make it into a Killing tensor field which descends to the given Killing tensor field on $\bh P^n$, we consecutively add to $L_d$ tensor fields $L_{d-m}$ which are the sections of the bundles $\Sym^m(\cV) \otimes \Sym^{d-m}(\cH)$ for $m=1, \dots, d$. These tensor fields must satisfy certain differential equations which we solve explicitly. To do that, we introduce certain differential and algebraic operators and study their commutators, similar to what we did in Subsection~\ref{ss:liftCPn}.

For $a, b \ge 0$, we define $\cS^{a,b}$ to be the space of smooth, symmetric, $\Sp(1)$-invariant, covariant tensor fields on $S^{4n+3}$ of rank $a+b$ such that $L(\xi^{k}, Y^l) = 0$ for any $L \in \cS^{a,b}$ and any vertical vector field $\xi$ and horizontal vector field $Y$, unless $k=a$ and $l = b$.

We define the following five operators:
\begin{equation}\label{eq:def4opH}
\begin{alignedat}{3}
  \delta: \, & \ocS^{a,b} \to \ocS^{a,b+1} & &\qquad\qquad& (\delta T)(\xi^a, Y^{b+1})& = (b+1) (\nabla_Y T)(\xi^a, Y^b),\\
  \iota: \, & \ocS^{a,b} \to \ocS^{a+1,b-1} & &\qquad\qquad& (\iota T)(\xi^{a+1}, Y^{b-1})& = \sum\nolimits_{i=1}^{4n} (\nabla_{e_i} T)(\xi^a, Y^{b-1},J_\xi e_i), \\
  \rJ: \, & \ocS^{a,b} \to \ocS^{a+1,b} & &\qquad\qquad& (\rJ T)(\xi^{a+1}, Y^b)& = b \, T (\xi^a, Y^{b-1},J_\xi Y), \\
  \tau: \, & \ocS^{a,b} \to \ocS^{a,b-2} & &\qquad\qquad& (\tau T)(\xi^a, Y^{b-2})& = \sum\nolimits_{i=1}^{4n} T (\xi^a, Y^{b-2},e_i, e_i), \\
  \nu: \, &  \ocS^{a,b} \to \ocS^{a,b} & &\qquad\qquad& (\nu T)(\xi^a, Y^b)& = \tbinom{a}{2} \|\xi\|^2\sum\nolimits_{\a=1}^{3} T (\xi_\a, \xi_\a,\xi^{a-2}, Y^b), \\
\end{alignedat}
\end{equation}
for $T \in \ocS^{a,b}$, where $\{e_i\}$ is an orthonormal horizontal frame, $\{\xi_\a\}$ is an orthonormal vertical frame, and $\xi$ and $Y$ are vertical and horizontal vector fields, respectively (and where we define $\ocS^{a,b} = 0$ when $b < 0$). Since the action of $\Sp(1)$ preserves the metric, the connection and the vertical and the horizontal distributions, all five operators above are well defined: they send $\Sp(1)$-invariant tensors to $\Sp(1)$-invariant tensors.

Suppose $K$ is a symmetric, covariant, $\Sp(1)$-invariant tensor field of rank $d$ on $S^{4n+3}$. We have the following decomposition:
\begin{equation}\label{eq:KLH}
  K=L_0 + L_1 + \dots + L_{d-1} + L_d, \quad \text{where  } L_{d-m} \in \cS^{m,d-m}.
\end{equation}

In terms of the tensor fields $L_m$, the Killing condition~\eqref{eq:defK} for $K$ is given as follows.
\begin{lemma} \label{l:KillingH}
  A symmetric, covariant, $\Sp(1)$-invariant tensor field $K$ given by~\eqref{eq:KLH} is Killing on $S^{4n+3}$ if and only if the following equations are satisfied:
  \begin{equation}\label{eq:KillingH}
    \delta L_{d-m} = 2m \, \rJ L_{d-m+1}, \qquad \text{for all  } m = 0, 1, \dots, d.
  \end{equation}
\end{lemma}
\begin{proof}
  Let $Y$ be a basic vector field on $S^{4n+3}$ and $\xi$ be a vertical, $\Sp(1)$-invariant vector field. As $K$ is $\Sp(1)$-invariant, we have $\xi(K((Y+\xi)^d))=0$, and hence the Killing condition $\nabla_{Y+\xi} K((Y+\xi)^d)=0$ for $K$ is equivalent to the equation $(\nabla_Y K) ((Y+\xi)^d) = 2d K((Y+\xi)^{d-1}, J_\xi Y)$, by O'Neill formulas~\eqref{eq:nablaH1}. Now from~\eqref{eq:KLH} we have $K ((Y+\xi)^d) = \sum_{m=0}^{d} \binom{d}{m} L_{d-m}(\xi^m, Y^{d-m})$ and $dK ((Y+\xi)^{d-1},J_\xi Y) = \sum_{m=0}^{d} \binom{d}{m} (d-m) L_{d-m}(\xi^m,$ $Y^{d-m-1},J_\xi Y)$, and equations~\eqref{eq:KillingH} follow, using the definitions of $\delta$ and $\rJ$ from~\eqref{eq:def4opH}.
\end{proof}

Suppose we are given a Killing tensor field $L$ of rank $d$ on $\bh P^n$. We pull it back to a tensor field $L_d =\pi^*L \in \cS^{0,d}$. Note that $\delta L_d = 0$, which is the first equation of~\eqref{eq:KillingH}. To prove Proposition~\ref{p:liftC}, we need to construct tensor fields $L_{d-1} \in \ocS^{1,d-1}, \, L_{d-2} \in \ocS^{2,d-2}, \ldots, L_{1} \in \ocS^{d-1,1}$ and $L_0 \in \ocS^{d,0}$ such that the remaining $d$ equation of~\eqref{eq:KillingH} are also satisfied, and then the tensor field $K$ given by~\eqref{eq:KLH} will be a Killing $\Sp(1)$-invariant tensor field which descends to $L$, as required. 

We prove the following:
\begin{lemma} \label{l:bracketsH}
  The operators~\eqref{eq:def4opH} satisfy the following commutation relations: for $T \in \ocS^{a,b}$,
  \begin{gather}
    2 ([\iota, \rJ] \, T)(\xi^{a+2}, Y^{b-1}) = \|\xi\|^2 ([\delta, \tau] \, T)(\xi^a, Y^{b-1}), \label{eq:iJH} \\
    ([\nu, \rJ] \, T)(\xi^{a+1}, Y^{b}) = a b \|\xi\|^2 \sum\nolimits_{\a=1}^{3} T(\xi^{a-1}, \xi_\a, Y^{b-1},J_\a Y), \label{eq:nuJH} \\
    [\tau, \rJ] = [\delta, \nu] = [\tau, \nu] = 0, \label{eq:tJH} \\
    [\delta, \iota] \, T =  (c(a,b) \rJ  - 4 [\nu, \rJ]) T, \quad \text{where} \quad c(a,b) = -4(2n+b-a-1), \label{eq:diH}
  \end{gather}
  and where $\xi$ and $Y$ are vertical and horizontal vector fields, respectively.
\end{lemma}
The proof of Lemma~\ref{l:bracketsH} is a direct, but quite lengthy, sequence of tensorial calculations. It is given in Subsection~\ref{ss:tensorial} below.

Assuming Lemma~\ref{l:bracketsH}, the proof of Proposition~\ref{p:liftH} is completed by the following lemma.

\begin{lemma} \label{l:lifts}
  The following holds.
  \begin{enumerate}[label=\emph{(\alph*)},ref=\alph*]
    \item \label{it:nuinv}
    For any $a > 0$ and $b \ge 0$, the operator $-\frac{2}{a}(\nu + a(2n+b)\, \id)$ on the space $\ocS^{a,b}$ is invertible.

    \item \label{it:Ld-m}
    Given a tensor filed $L_d \in \ocS^{0,d}$ with $\delta L_d = 0$, the tensor fields $L_{d-m} \in \ocS^{m,d-m}$, $m = 1, \dots, d$, consecutively defined as below satisfy equations~\eqref{eq:KillingH}:
    \begin{align}
        L_{d-1} & = - (2(2n+d-1))^{-1} \iota L_d, \label{eq:Ld-1} \\
        -\tfrac{2}{m}(\nu + m(2n+d-m))L_{d-m} & = \iota L_{d-m+1} - (m-1)\|\xi\|^2 \tau L_{d-m+2}, \; \; m =1,\dots,d \label{eq:Ld-m}.
    \end{align}
  \end{enumerate}
\end{lemma}
The tensor field $L_{d-m}$ in~\eqref{eq:Ld-m} is well defined, by assertion~\eqref{it:nuinv}.

\begin{proof}
  We note that it is sufficient to prove assertion~\eqref{it:nuinv} for the action of $\nu$ on the space of tensors $\ocS^{a,0}$ at a single point $p$. The latter is linearly isomorphic to the space of homogeneous polynomials $\cP^a$ of degree $a$ in three variables $(\eta_1,\eta_2,\eta_3)$, where we put $\xi= \sum_{\a=1}^{3} \eta_\a \xi_\a$ at the point $p$. Under this linear isomorphism, the action of $\nu$ on $\ocS^{a,0}$ corresponds to the action $\tilde{\nu}(f) = \frac12 \|\eta\|^2 \Delta_\eta f$ for $f \in \cP^a$. The space $\cP^a$ has the direct sum decomposition $\cP^a = \cH^a \oplus \|\eta\|^2 \cH^{a-2} \oplus \|\eta\|^4 \cH^{a-4} \oplus \dots$, where $\cH^k \subset \cP^k$ is the space of harmonic homogeneous polynomials of degree $k$. Each subspace $\|\eta\|^{2k} \cH^{a-2k},\, 0 \le k \le a/2$, in this decomposition is an eigenspace of the operator $\tilde{\nu}$ with eigenvalue $k(2k+2a-3)$. As any such eigenvalue is non-negative, the operator $\tilde{\nu} + a(2n+b)\,\id$ is invertible on $\cP^a$, and hence the operator $\nu+ a(2n+b)\id$ is invertible on $\ocS^{a,b}$, for all $a>0,  \, b \ge 0$.

  We now prove assertion~\eqref{it:Ld-m}.

  To check equation~\eqref{eq:Ld-1} we note that $[\nu, \rJ] L_d = 0$ by~\eqref{eq:nuJH} and $\delta L_d = 0$ by assumption. Then from~\eqref{eq:diH} we find $\delta \iota L_d = ([\delta,\iota] + \iota \delta ) L_d = - 4(2n+d-1) \rJ L_d$, and so $L_{d-1}$ defined by~\eqref{eq:Ld-1} satisfies the equation $\delta L_{d-1} = 2 \rJ L_d$, as required.

  Let $k > 1$, and assume that equations~\eqref{eq:KillingH} are satisfied by the tensor fields $L_{d-m}$ defined by~\eqref{eq:Ld-1} and~\eqref{eq:Ld-m} for $m < k$. Acting by $\delta$ on the right-hand side of equation~\eqref{eq:Ld-m} we obtain

  \begin{equation*}
  \begin{split}
    & \delta (\iota L_{d-k+1} - (k-1)\|\xi\|^2 \tau L_{d-k+2}) \\
    &= ([\delta, \iota] + \iota \delta) L_{d-k+1} - (k-1)\|\xi\|^2 ([\delta, \tau] + \tau \delta) L_{d-k+2} \\
    &
    {\begin{aligned}
    = (c(k-1,d-k+1) \rJ - 4 [\nu, \rJ]) L_{d-k+1} & + 2(k-1) \iota \rJ L_{d-k+2} \\
    & - (k-1)(2 [\iota, \rJ] + \|\xi\|^2 \tau \delta) L_{d-k+2}
    \end{aligned}}
    \\
    & = (c(k-1,d-k+1) \rJ - 4 [\nu, \rJ]) L_{d-k+1} + 2(k-1) \rJ (\iota L_{d-k+2} - 2(k-2) \|\xi\|^2 \tau L_{d-k+3}) \\
    & = (c(k-1,d-k+1) \rJ - 4 [\nu, \rJ]) L_{d-k+1} - 4 \rJ (\nu + (k-1)(2n+d-k+1)\id) L_{d-k+1} \\
    & = - 4 (\nu + k(2n + d - k)\id) \rJ L_{d-k+1},
  \end{split}
  \end{equation*}
  where in the first equation we used~\eqref{eq:diH} and~\eqref{eq:iJH}, in the second equation, \eqref{eq:KillingH} with $m=k-2$ and the fact that $[\tau, \rJ] = 0$ from~\eqref{eq:tJH}, in the third, equation~\eqref{eq:Ld-m} for $m = k-2$, and in the fourth, the formula for $c(k-1,d-k+1)$ given in~\eqref{eq:diH}.

  Since $\delta$ commutes with $\nu$ by~\eqref{eq:tJH}, we now obtain from~\eqref{eq:Ld-m} with $m=k$ that $\delta L_{d-k} = 2k \rJ L_{d-k+1}$, as required.
\end{proof}

\begin{remark} \label{rem:higher}
  If one wishes to write down an explicit formula for $L_{d-m}, \, m > 1$, one can find the operator inverse to the operator $-\tfrac{2}{m}(\nu + m(2n+d-m))$ in~\eqref{eq:Ld-m} on $\ocS^{m,d-m}$. This can be done directly, as from the proof of assertion~\eqref{it:nuinv}, we know the minimal polynomial of the action of $\nu$ on $\ocS^{m,d-m}$.
\end{remark}

%

\subsection{Proof of Proposition~\ref{p:polyH}}
\label{ss:decoHPn}

  The proof goes somewhat similar to that of Proposition~\ref{p:polyC}: after complexification, we apply the First Fundamental Theorem of Invariant Theory to the sum of $2n+2$ copies of the standard modules $\bc^4$ of the group $\SO(4,\bc) \cong \SL(2,\bc) \times \Sp(1,\bc)/\mathbb{Z}_2$.

  In the space $\br^{8n+8} = \br^{4n+4} \oplus \br^{4n+4} = \{(X, P) \, | \, X, P \in \br^{4n+4}\}$, introduce orthonormal bases $\{e_i\}$ and $\{f_i\},\, i=1, \dots, 4n+4$, for the two copies of the space $\br^{4n+4}$ as in Subsection~\ref{ss:sphere}. We specify the bases $\{e_i\}$ and $\{f_i\}$ in such a way that relative to them, the operators $J_\alpha, \, \alpha = 1, 2, 3$, are given by the matrices $\diag(\rL_\alpha, \dots, \rL_\alpha)$ ($n+1$ diagonal $4 \times 4$ blocks), where
  \begin{equation*}
    \rL_1 = \begin{pmatrix}
                0 & -1 & 0 & 0 \\
                1 & 0 & 0 & 0 \\
                0 & 0 & 0 & -1 \\
                0 & 0 & 1 & 0 \\
              \end{pmatrix}, \quad
    \rL_2 = \begin{pmatrix}
                0 & 0 & -1 & 0 \\
                0 & 0 & 0 & 1 \\
                1 & 0 & 0 & 0 \\
                0 & -1 & 0 & 0 \\
              \end{pmatrix}, \quad
    \rL_3 = \begin{pmatrix}
                0 & 0 & 0 & -1 \\
                0 & 0 & -1 & 0 \\
                0 & 1 & 0 & 0 \\
                1 & 0 & 0 & 0 \\
              \end{pmatrix}
  \end{equation*}
  (so that $\rL_1, \rL_2$ and $\rL_3$ are the matrices of the left multiplications by the imaginary quaternions $\ri, \mathrm{j}$ and $\mathrm{k}$, respectively). Similarly, we introduce the matrices of the right multiplications by
  \begin{equation*}
    \rR_1 = \begin{pmatrix}
                0 & -1 & 0 & 0 \\
                1 & 0 & 0 & 0 \\
                0 & 0 & 0 & 1 \\
                0 & 0 & -1 & 0 \\
              \end{pmatrix}, \quad
    \rR_2 = \begin{pmatrix}
                0 & 0 & -1 & 0 \\
                0 & 0 & 0 & -1 \\
                1 & 0 & 0 & 0 \\
                0 & 1 & 0 & 0 \\
              \end{pmatrix}, \quad
    \rR_3 = \begin{pmatrix}
                0 & 0 & 0 & -1 \\
                0 & 0 & 1 & 0 \\
                0 & -1 & 0 & 0 \\
                1 & 0 & 0 & 0 \\
              \end{pmatrix}.
  \end{equation*}
  Clearly $[\rL_\alpha, \rR_\beta] = 0$, and also $\rL_1\rL_2=\rL_3$, but $\rR_1\rR_2=-\rR_3$.

  The action of the group $\Sp(1)$ on the space $\br^{8n+8}$ is given by
  \begin{equation}\label{eq:Sp1action}
  U.(X,P) = (UX, UP), \text{ for } U=u_0 I_{4n}+u_1 J_1 + u_2 J_2 + u_3 J_3 \in \Sp(1),
  \end{equation}
  where $\sum_{\alpha=0}^{3} u_\alpha^2 = 1$. In particular, the space $\cL$ of polynomials $T_A$ given by~\eqref{eq:hvf} is spanned by the polynomials
  \begin{equation}\label{eq:KvfH}
    W_{jk,\alpha} = \<\rR_\alpha x^{(j)}, p^{(k)}\> + \<\rR_\alpha x^{(k)}, p^{(j)}\>, \quad \text{and} \quad W_{jk,0} = \<x^{(j)}, p^{(k)}\> - \<x^{(k)}, p^{(j)}\>,
  \end{equation}
  for $1 \le j, k \le n+1,\; \alpha =1,2,3$, where $x^{(j)}=(x_{4j-3},x_{4j-2},x_{4j-1},x_{4j})^t$ and $p^{(j)}=(p_{4j-3},p_{4j-2},p_{4j-1}, p_{4j})^t$.

  The actions of $\SL(2)$ and $\Sp(1)$ on $\br^{8n+8}$ defined by~\eqref{eq:SL2action} and~\eqref{eq:Sp1action} commute, and so we need to prove that the algebra of polynomials on $\br^{8n+8}$ invariant under the action of the group $(\SL(2) \times \Sp(1))/\mathbb{Z}_2$ given by~\eqref{eq:SL2action} and~\eqref{eq:Sp1action} is generated by the polynomials $W_{jk,\alpha}$ and $W_{jk,0}$ defined by~\eqref{eq:KvfH} and the polynomials $T_{S_1,S_2}$ defined by \eqref{eq:hTS}. Here $\mathbb{Z}_2 = \{(I_2,I_2),(-I_2,-I_2)\}$ is the ineffective kernel of the action.

  We complexify the space $\br^{8n+8}$ and consider the action of the group $\SL(2,\bc) \times \Sp(1,\bc)$ on the space $\bc^{8n+8}$ defined by the same formulas~\eqref{eq:SL2action} and~\eqref{eq:Sp1action}, but with complex coefficients. We note that $\Sp(1,\bc) \cong \SL(2,\bc)$ and that $(\SL(2,\bc) \times \SL(2,\bc))/\mathbb{Z}_2 \cong \SO(4,\bc)$.

  Introduce a new basis $\{E_i,F_i\},\, i = 1, \dots, 4n+4$, for $\bc^{8n+8}$ by defining
  \begin{equation*}
  \begin{pmatrix} E_{4j-3} \\ E_{4j-2} \\ E_{4j-1} \\ E_{4j} \\ F_{4j-3} \\ F_{4j-2} \\ F_{4j-1} \\ F_{4j} \end{pmatrix} =
  \begin{pmatrix}
    1 & \ri & 0 & 0 & 0 & 0 & 1 & -\ri \\
    -\ri & 1 & 0 & 0 & 0 & 0 & \ri & 1 \\
    0 & 0 & 1 & -\ri & -1 & -\ri & 0 & 0 \\
    0 & 0 & \ri & 1 & \ri & -1 & 0 & 0 \\
    0 & 0 & -1 & -\ri & 1 & -\ri & 0 & 0 \\
    0 & 0 & \ri & -1 & \ri & 1 & 0 & 0 \\
    1 & -\ri & 0 & 0 & 0 & 0 & 1 & \ri \\
    \ri & 1 & 0 & 0 & 0 & 0 & -\ri & 1
  \end{pmatrix}
  \begin{pmatrix} e_{4j-3} \\ e_{4j-2} \\ e_{4j-1} \\ e_{4j} \\ f_{4j-3} \\ f_{4j-2} \\ f_{4j-1} \\ f_{4j} \end{pmatrix},
  \end{equation*}
  for $j=1,\dots, n+1$. The corresponding coordinate change is given by
  \begin{equation}\label{eq:xptozw}
    z^{(j)} = \tfrac14(I_4 + \ri R_1) x^{(j)} -\tfrac14(R_2 + \ri R_3) p^{(j)}, \quad w^{(j)} = \tfrac14(R_2 - \ri R_3) x^{(j)} + \tfrac14(I_4 - \ri R_1) p^{(j)},
  \end{equation}
  where $\sum_{i=1}^{4n+4}(x_ie_i+p_if_i)=\sum_{i=1}^{4n+4}(z_iE_i+w_iF_i)$ and where $z^{(j)}=(z_{4j-3},z_{4j-2},z_{4j-1},z_{4j})^t$ and $w^{(j)}=(w_{4j-3},w_{4j-2},w_{4j-1}, w_{4j})^t$, for $j=1, \dots, n+1$.
  One can check directly that the subspaces $\Span_\bc(E_{4j-3}, E_{4j-2}, E_{4j-1}, E_{4j})$ and $\Span_\bc(F_{4j-3}, F_{4j-2}, F_{4j-1}, F_{4j})$ are $\SO(4,\bc)$-invariant. Moreover, the action of any element of $\SO(4,\bc)$ on each of these subspaces, relative to the chosen bases, is given by the same matrix, and this action preserves the quadratic form $\<z^{(j)}, z^{(j)}\> = z^2_{4j-3}+z^2_{4j-2}+z^2_{4j-1}+z^2_{4j}$ (respectively, $\<w^{(j)}, w^{(j)}\> = w^2_{4j-3}+w^2_{4j-2}+w^2_{4j-1}+w^2_{4j}$), where we continue using the same notation $\ip$ for the complex bilinear extension of the inner product from $\br^4$ to $\bc^4$.

  Thus we have the group $\SO(4,\bc)$ simultaneously acting on $2n+2$ copies of its standard representation on $\bc^4$, or equivalently, acting by the left multiplication on the $4 \times (2n+2)$ matrix $Z=(z^{(1)} \,| \, w^{(1)} \,| \dots | z^{(n+1)} \,| \, w^{(n+1)})$. By the First Fundamental Theorem of Invariant Theory~\cite[Theorem~11.2.1]{P}, the algebra of invariant polynomials of this action is generated by inner products of the pairs of columns of $Z$ and by $4 \times 4$ determinants formed by the quadruples of columns of $Z$. For the inner products we obtain, using~\eqref{eq:xptozw} and~\eqref{eq:KvfH}:
  \begin{gather*}
    \<z^{(j)}, z^{(k)}\> = \tfrac18 (W_{jk,2} + \ri W_{jk,3}), \qquad \<w^{(j)}, w^{(k)}\> = \tfrac18 (W_{jk,2} - \ri W_{jk,3}), \\
    \<z^{(j)}, w^{(k)}\> = \tfrac18 (W_{jk,0} + \ri W_{jk,1}), \qquad \<z^{(k)}, w^{(j)}\> = \tfrac18 (-W_{jk,0} + \ri W_{jk,1}),
  \end{gather*}
  and so the space of the inner products of the pairs of columns of $Z$ is spanned by the polynomials $\{W_{jk,\alpha}, W_{jk,0}\}$, and hence is the complexification of the space $\cL$, by~\eqref{eq:KvfH}.

  For the $4 \times 4$ determinants, we use the following easy-to-check identity: for any four vectors $y^1,y^2,y^3,y^4 \in \bc^4$, one has
  \begin{equation*}
  \det(y^1 \, | \, y^2 \, | \, y^3 \, | \, y^4) =\<y^1, y^4\>\<y^2, y^3\> - \<y^1, y^3\>\<y^2, y^4\> + \sum\nolimits_{\alpha=1}^{3} \<y^1, \rL_\alpha y^2\>\<y^3, \rL_\alpha y^4\>.
  \end{equation*}
  It follows that modulo the inner products, the $4 \times 4$ determinants of the matrix $Z$ are given by the expressions $\sum_{\alpha=1}^{3} \<y^1, \rL_\alpha y^2\>\<y^3, \rL_\alpha y^4\>$, where each $y^s$ is either some $z^{(j)}$ or some $w^{(k)}$. Using~\eqref{eq:xptozw} we get for $j \ne k$:
  \begin{gather*}
    \<z^{(j)}, \rL_\a z^{(k)}\> = \frac18 \Big\langle \begin{pmatrix} 0 & R_2 - \ri R_3 \\ -R_2 + \ri R_3 & 0 \end{pmatrix} \begin{pmatrix} \rL_\alpha &  0 \\ 0 & \rL_\alpha \end{pmatrix} \begin{pmatrix} x^{(j)} \\ x^{(k)} \end{pmatrix}, \begin{pmatrix} p^{(j)} \\ p^{(k)} \end{pmatrix} \Big\rangle = \<S^{(1)}_{jk} J_\alpha X, P\>,\\
    \<w^{(j)}, \rL_\a w^{(k)}\> = \frac18 \Big\langle \begin{pmatrix} 0 & R_2 + \ri R_3 \\ -R_2 - \ri R_3 & 0 \end{pmatrix} \begin{pmatrix} \rL_\alpha &  0 \\ 0 & \rL_\alpha \end{pmatrix} \begin{pmatrix} x^{(j)} \\ x^{(k)} \end{pmatrix}, \begin{pmatrix} p^{(j)} \\ p^{(k)} \end{pmatrix} \Big\rangle = \<S^{(2)}_{jk} J_\alpha X, P\>, \\
    \<z^{(j)}, \rL_\a w^{(k)}\> = \frac18 \Big\langle \begin{pmatrix} 0 & -I_4 + \ri R_1 \\ -I_4 - \ri R_1 & 0 \end{pmatrix} \begin{pmatrix} \rL_\alpha &  0 \\ 0 & \rL_\alpha \end{pmatrix} \begin{pmatrix} x^{(j)} \\ x^{(k)} \end{pmatrix}, \begin{pmatrix} p^{(j)} \\ p^{(k)} \end{pmatrix} \Big\rangle = \<S^{(3)}_{jk} J_\alpha X, P\>, \\
    \<z^{(j)}, \rL_\a w^{(j)}\> = -\frac14 \<\rL_\alpha x^{(j)}, p^{(j)}\> = \<S_j J_\alpha X, P\>,
  \end{gather*}
  where the matrix $S^{(1)}_{jk}$ has the matrices $\frac18(R_2 - \ri R_3)$ and $-\frac18(R_2 - \ri R_3)$ in its $(j,k)$-th and $(k,j)$-th $4 \times 4$ blocks, respectively, and zeros elsewhere, and the matrices $S^{(2)}_{jk}$ and $S^{(3)}_{jk}$ are constructed similarly, with the blocks given by the formula above. The matrix $S_j$ has the matrix $-\frac14 I_4$ in its $j$-th diagonal $4 \times 4$ block and zeros elsewhere. Note that all the matrices $S^{(1)}_{jk}, S^{(2)}_{jk},S^{(3)}_{jk},\, j \ne k$, and $S_j$ are symmetric and commute with $J_1,J_2$ and $J_3$, and hence all of them belong to $V^\bc_{n+1}$. Hence any $4 \times 4$ determinant of the matrix $Z$, modulo a quadratic form in the inner products, has the form $T_{S_1,S_2}$, for some $S_1,S_2 \in V^\bc_{n+1}$, as given by~\eqref{eq:hTS}.

  We deduce that the algebra of complex polynomials on $\bc^{8n+8}$ invariant under the action of the group $\SL(2,\bc) \times \Sp(1,\bc)$ given by~(\ref{eq:SL2action}, \ref{eq:Sp1action}) is contained in the algebra generated by the polynomials $W_{jk,\alpha}$ and $W_{jk,0}$ defined by~\eqref{eq:KvfH} and the polynomials $T_{S_1,S_2}$, defined by \eqref{eq:hTS} for $S_1, S_2 \in V^\bc_{n+1}$. As all these polynomials are $(\SL(2,\bc) \times \Sp(1,\bc))$-invariant, the two algebras coincide. Now any real polynomial invariant under the action of the group $\SL(2) \times \Sp(1)$ is also invariant under the action of the group $\SL(2,\bc) \times \Sp(1,\bc)$ and so belongs to the complex algebra generated by the polynomials $W_{jk,\alpha}, \, W_{jk,0}$ and $T_{S_1,S_2}$. As all these polynomials are real, it also belongs to the real algebra generated by them.

  This completes the proof of Proposition~\ref{p:polyH}.

  \begin{remark} \label{rem:SFT}
    The relations between the invariant polynomials of the standard action of the group $\SO(4,\bc)$ on the $n+1$ copies of $\bc^4$ which we used in the proof are given by the
    Second Fundamental Theorem of Invariant Theory~\cite[end of Section~17]{W}. They are generated by three families of polynomials. Firstly, the product of any two $4 \times 4$ determinants is a fourth degree polynomial in the inner products. Secondly, the Gram matrix of any $5$ elements of $\bc^4$ is singular. And thirdly, evaluation of the exterior product of any $5$ elements of $\bc^4$ on any other element gives a zero.

    In application to Killing tensor fields on $\bh P^n$, we also note that the map $K \mapsto \overline{K}$ (introduced in Subsection~\ref{ss:subm}) obviously has a nontrivial kernel: the ideal generated by the $1$-forms dual to $\xi_1, \xi_2$ and $\xi_3$.
  \end{remark}

\subsection{Tensorial calculations. Proof of Lemma~\ref{l:bracketsH}}
\label{ss:tensorial}

In this subsection, we verify the commutation relations in Lemma~\ref{l:bracketsH}.

It will be much easier to work with $\Sp(1)$-invariant vertical vector fields. So throughout this subsection only, we adopt the following notation convention (which differs from that in the rest of Sections~\ref{s:HPn}). We choose the vertical orthonormal frame $\xi_1, \xi_2, \xi_3$ to be $\Sp(1)$-invariant (there are many such choices: effectively, we choose such frame separately on every fiber $S^3$). We keep the notation $J_\a$ for $J_{\xi_\a}, \, \a =1, 2, 3$, and we specify the orientation of the vertical basis $\{\xi_1, \xi_2, \xi_3\}$ in such a way that $J_1 J_2 = J_3$.

Note that all the definitions~\eqref{eq:def4opH} and all the identities in Lemma~\ref{l:bracketsH} are tensorial, and hence do not depend on a particular choice of the vector fields.

\subsubsection{O'Neill formulas}
\label{sss:O'Neill}

For vertical $\Sp(1)$-invariant vector fields $\xi, \eta$ and horizontal $\Sp(1)$-invariant vector fields $Y, Z$, then $(\nabla_Y Z)^\mathcal{H}$ is basic and the following holds:
\begin{equation}\label{eq:nablaH1}
\begin{gathered}
  \nabla_Y Z = \left(\nabla_Y Z \right)^\mathcal{H} -\sum\nolimits_{\alpha=1}^{3}\<J_\alpha Y,Z\>\xi_\alpha, \\
  \nabla_Y \xi = J_\xi Y + (\nabla_Y \xi)^\mathcal{V}, \qquad \nabla_\xi \eta = \nabla^F_\xi \eta, \qquad \nabla_\xi Y = J_\xi Y,
\end{gathered}
\end{equation}
where $\nabla, \nabla^*$ and $\nabla^F$ are the Levi-Civita connections on $S^{4n+3}, \, \bh P^n$ and the fiber $S^3$, respectively, and the superscript $V$ denotes the vertical component and $\nabla^*_{\pi_* Y} \pi_* Z = \pi_*(\nabla_Y Z)$. Note that for vertical $\Sp(1)$-invariant vector fields, $\nabla^F$ is given by Koszul formula; assuming the $\Sp(1)$-invariant vector fields $\xi_1, \xi_2, \xi_3$ are chosen in such a way that $J_1J_2=J_3$, we have
\begin{equation}\label{eq:nablaxixi}
  \nabla_{\xi_\a} \xi_\a = 0, \qquad \nabla_{\xi_\a} \xi_\b = \xi_\gamma, \qquad \nabla_{\xi_\a} \xi_\gamma = - \xi_\b,
\end{equation}
where $(\a,\b,\gamma)$ is a cyclic permutation of $(1,2,3)$.

From~\eqref{eq:nablaH1} we obtain that for a tensor field $T \in \ocS^{a,b}$ and $\Sp(1)$-invariant horizontal vector fields $Y, Z$ and vertical vector field $\eta$, we have
\begin{align}\label{eq:nablaTap1}
  (\nabla_Z T)(\eta^{a+1},Y^{b-1}) & = -(a+1) T(\eta^a, J_\eta Z,Y^{b-1}), \\
  (\nabla_Z T)(\eta^{a-1},Y^{b+1}) & = (b+1) \sum\nolimits_{\a=1}^{3} T(\eta^{a-1}, \xi_\a,Y^b) \<J_\a Z,Y\>. \label{eq:nablaTam1}
\end{align}

\subsubsection{Specifying vector fields at a point}
\label{sss:point}

We simplify the subsequent calculations by showing that all the vectors we need at a point can be extended to $\Sp(1)$-invariant vector fields in a neighbourhood in such a way that all non-tensorial components of their derivatives at that point in O'Neill formulas are zeros.

\begin{lemma} \label{l:atapoint}
  Let $p \in S^{4n+3}$, and let $\{e_i^0\}$ and $\{\xi_\a^0\}$ be orthonormal horizontal and vertical bases in $T_pS^{4n+3}$, respectively. Let $\xi^0$ and $Y^0$ be a vertical and a horizontal vector in $T_pS^{4n+3}$, respectively. Then on a neighbourhood of $p$, there exist an orthonormal, $\Sp(1)$-invariant horizontal frame $\{e_i\}$, an orthonormal, $\Sp(1)$-invariant vertical frame $\{\xi_\a\}$, an $\Sp(1)$-invariant vertical vector field $\xi$ and an $\Sp(1)$-invariant horizontal vector field $Y$ such that
  \begin{equation}\label{eq:atapoint}
    \begin{gathered}
    e_i(p)=e_i^0, \qquad \xi_\a(p) = \xi_\a^0, \qquad Y(p)=Y^0, \qquad \xi(p) = \xi^0 \qquad \text{and} \\
    (\nabla_Z Y (p))^\mathcal{H} = (\nabla_Z e_i (p))^\mathcal{H} = 0, \qquad (\nabla_Z \xi (p))^\mathcal{V} = (\nabla_Z \xi_\a (p))^\mathcal{V} = 0,  \\
    (\nabla_Z (J_{\a} e_i) (p))^\mathcal{H} = (\nabla_Z (J_{\a} Y) (p))^\mathcal{H} = (\nabla_Z (J_\xi e_i) (p))^\mathcal{H} = (\nabla_Z (J_\xi Y) (p))^\mathcal{H} = 0,
    \end{gathered}
  \end{equation}
  for any horizontal $Z \in T_pS^{4n+3}$.
\end{lemma}
\begin{proof} 
  Let $x^i, \, i =1, \dots, 4n$, be geodesic normal coordinates on $\bh P^n$ centered at the point $o=\pi(p)$ such that $\frac{\db}{\db x^i}(o)=\pi_* e_i(p)$. The metric tensor of $\bh P^n$ on a neighbourhood of $o$ is given by $g_{ij}(x) = \delta_{ij} - \frac13 R^*_{iklj} x^kx^l + \dots$. Define a symmetric matrix $S(x)=S_{ij}(x)$ by $S(x)=g^{-1/2}(x)$, so that $S_{ij}(x) = \delta_{ij} + \frac16 R^*_{iklj} x^kx^l + \dots$. Then the vector fields $e_i^* = \sum_{j=1}^{4n} S_{ij}(x) \frac{\db}{\db x^j}$ are orthonormal and $\nabla^*_{Z^*} e_i^*(p) = 0$, for any horizontal vector $Z^* \in T_o \bh P^n$. Hence the basic vector fields $e_i$ obtained by lifting $e_i^*, \, i=1, \dots, n$, satisfy the equations $e_i(p)=e_i^0$ and $(\nabla_Z e_i (p))^\mathcal{H} = 0$, as required. Moreover, the (basic) vector field $Y = \sum_{i=1}^{4n} \<Y^0,e_i^0\> e_i$ satisfies the equations $Y(p)=Y^0$ and $(\nabla_Z Y (p))^\mathcal{H} = 0$.

  To construct the vector fields $\xi_\a$, we take an arbitrary smooth $\Sp(1)$-invariant vertical frame $\{\eta_\a\}$ on a neighbourhood of $p$ such that $\eta_\a(p) = \xi_\a^0$ and denote $\theta_\a^\b(Z) = \<\nabla_Z \eta_\a,\eta_\b\>$ for a horizontal vector field $Z$. At the point $p$, we denote $K_i \in \so(3)$ the constant matrix defined by $(K_i)_\a^\b = \theta_\a^\b(e_i)(p)$. Then the $\Sp(1)$-invariant, orthonormal vertical vector fields $\xi_\a = \sum_{\a=1}^{3} (\exp(-\sum_{i=1}^{4n} (x^i \circ \pi)K_i))_\a^\beta \eta_\beta$ satisfy the equations $\xi_\a(p) = \xi_\a^0$ and $(\nabla_Z \xi_\a (p))^\mathcal{V} = 0$, as required. Similarly to the above, we can now take $\xi = \sum_{\a=1}^{3} \<\xi^0,\xi_\a^0\> \xi_\a$.

  To prove the equations in the last line of~\eqref{eq:atapoint}, we assume that $X_1,X_2$ and $X_3$ are horizontal, $\Sp(1)$-invariant vector fields on neighbourhood of $p$ such that $(\nabla_Z X_i(p))^\mathcal{H}=0$, and $\eta$ is a vertical, $\Sp(1)$-invariant vector field such that $(\nabla_Z \eta(p))^\mathcal{V}=0$, for any horizontal vector field $Z$. From~\eqref{eq:nablaH1}, at the point $p$ we get $(\nabla_{X_1}\nabla_{X_2}\eta)^\mathcal{H} = (\nabla_{X_1}(J_\eta X_2))^\mathcal{H}$, as $(\nabla_{X_2}\eta(p))^\mathcal{V} = 0$. Moreover, by~\eqref{eq:nablaH1} at the point $p$ we have $(\nabla_{\nabla_{X_1}X_2}\eta)^\mathcal{H} = J_\eta ((\nabla_{X_1}X_2)^\mathcal{H}) + (\nabla_{(\nabla_{X_1}X_2)^\mathcal{V}}\eta)^\mathcal{H}= 0$, which gives $(\nabla_{X_1}\nabla_{X_2}\eta - \nabla_{\nabla_{X_1}X_2}\eta)^\mathcal{H} = (\nabla_{X_1}(J_\eta X_2))^\mathcal{H}$. Subtracting the same expression, with $X_1$ and $X_2$ interchanged we obtain $(R(X_1,X_2)\eta)^\mathcal{H} = (\nabla_{X_1}(J_\eta X_2) - \nabla_{X_2}(J_\eta X_1))^\mathcal{H}$. As $R(X_1,X_2,\eta,X_3) = 0$ we get $0=\<\nabla_{X_1}(J_\eta X_2),X_3\> - \<\nabla_{X_2}(J_\eta X_1),X_3\> = X_1\<J_\eta X_2,X_3\> - X_2 \<J_\eta X_1,X_3\>$ at $p$. It follows that the trilinear form $(X_1,X_2,X_3) \mapsto X_1\<J_\eta X_2,X_3\>(p)$ is symmetric by the first two arguments and is antisymmetric by the second two, from which $X_1\<J_\eta X_2,X_3\>(p) = 0$, and so $\<\nabla_{X_1}(J_\eta X_2),X_3\>(p) = 0$. Hence $(\nabla_{X_1}(J_\eta X_2)(p))^\mathcal{H} = 0$, and the equations in the last line of~\eqref{eq:atapoint} follow.
\end{proof}

By Lemma~\ref{l:atapoint} and from~\eqref{eq:nablaH1}, in all tensorial calculations below, we can assume that all the vertical vector fields $\eta = \xi, \xi_\a$ and all the horizontal vector fields $X=e_i, Y, J_\eta e_i, J_\eta Y$ are chosen to be $\Sp(1)$-invariant and to satisfy the equations
\begin{equation}\label{eq:nablaHatp}
\nabla_Z \eta = J_\eta Z, \qquad \nabla_Z X = - \sum\nolimits_{\a=1}^{3}\<J_\a Z, X\>\xi_\a,
\end{equation}
at a point $p \in S^{4n+3}$, for any horizontal $Z \in T_pS^{4n+3}$.

\subsubsection{Proof of identities~\eqref{eq:iJH}, \eqref{eq:nuJH} and \eqref{eq:tJH} of Lemma~\ref{l:bracketsH}}
\label{sss:firstthreecomm}

We take a point $p \in S^{4n+3}$ and assume that in a neighbourhood of $p$, all the vector fields below are chosen as in the last paragraph of Subsection~\ref{sss:point}.

We start with~\eqref{eq:tJH}. The fact that $[\tau, \rJ] = [\tau, \nu] = 0$ easily follows from the definition~\eqref{eq:def4opH}.

To see that $[\delta, \nu]=0$ we note that $\delta \nu T, \, \nu \delta T \in \ocS^{a,b+1}$ for $T  \in \ocS^{a,b}$, and so we compute: $(\delta (\nu T))(\xi^a, Y^{b+1}) = (b+1) (\nabla_Y (\nu T))(\xi^a, Y^b) = (b+1) Y ((\nu T)(\xi^a, Y^b)) = (b+1) \tbinom{a}{2}  \sum\nolimits_{\a=1}^{3} Y(\|\xi\|^2 T (\xi_\a^2, \xi^{a-2}, Y^b)) = (b+1) \tbinom{a}{2} \|\xi\|^2 \sum\nolimits_{\a=1}^{3} (\nabla_Y T) (\xi_\a^2, \xi^{a-2}, Y^b)$. Reversing the order we get: $(\nu (\delta T))(\xi^a, Y^{b+1}) = \tbinom{a}{2} \|\xi\|^2 \sum\nolimits_{\a=1}^{3} (\delta T) (\xi_\a^2, \xi^{a-2}, Y^b) = (b+1) \times$ $\times \tbinom{a}{2} \|\xi\|^2 \sum\nolimits_{\a=1}^{3} (\nabla_Y T) (\xi_\a^2, \xi^{a-2}, Y^b)$, as required.

For~\eqref{eq:nuJH}, we have $\nu \rJ T, \, \rJ \nu T \in \ocS^{a+1,b}$ for $T \in \ocS^{a,b}$, and so $(\nu (\rJ \, T))(\xi^{a+1}, Y^{b}) = \tbinom{a+1}{2} \times$  $\times \|\xi\|^2 \sum\nolimits_{\a=1}^{3} (\rJ \, T) (\xi_\a^2, \xi^{a-1}, Y^b) = \tbinom{a+1}{2} b \|\xi\|^2 \sum\nolimits_{\a=1}^{3} \tfrac{1}{a+1}((a-1)T (\xi_\a^2, \xi^{a-2}, Y^{b-1},J_\xi Y) + 2T (\xi_\a, \xi^{a-2}, Y^b,J_{\xi_\a} Y))$. On the other hand, $(\rJ \, (\nu T))(\xi^{a+1}, Y^{b}) = b (\nu T)(\xi^a, Y^{b-1},J_\xi Y) = b \tbinom{a}{2} \|\xi\|^2 \sum\nolimits_{\a=1}^{3} T(\xi_\a^2,\xi^{a-2}, Y^{b-1},J_\xi Y)$, and the claim follows.

To prove~\eqref{eq:iJH} we note that for $T  \in \ocS^{a,b}$, one has $\iota \rJ \, T, \, \rJ \iota \, T \in \ocS^{a+2,b-1}$ and $\delta \tau  T, \, \tau \delta \, T \in \ocS^{a,b-1}$. We compute:
\begin{align*}
  (\iota (\rJ \, T)) & (\xi^{a+2}, Y^{b-1}) = \sum\nolimits_{i=1}^{4n} (\nabla_{e_i} (\rJ \, T))(\xi^{a+1}, Y^{b-1},J_\xi e_i) \\
   & = \sum\nolimits_{i=1}^{4n} e_i((\rJ \, T))(\xi^{a+1}, Y^{b-1},J_\xi e_i))\\
   &= \sum\nolimits_{i=1}^{4n} e_i((b-1) T(\xi^a, Y^{b-2},J_\xi Y, J_\xi e_i) - \|\xi\|^2 T(\xi^a, Y^{b-1}, e_i)) \\
   & = \sum\nolimits_{i=1}^{4n} ((b-1) (\nabla_{e_i} T)(\xi^a, Y^{b-2},J_\xi Y, J_\xi e_i) - \|\xi\|^2 (\nabla_{e_i} T)(\xi^a, Y^{b-1}, e_i)), \\
  (\rJ \, (\iota T)) & (\xi^{a+2}, Y^{b-1}) = (b-1) (\iota T)(\xi^{a+1}, Y^{b-2},J_\xi Y) \\
  & = (b-1) \sum\nolimits_{i=1}^{4n} (\nabla_{e_i} T)(\xi^a, Y^{b-2},J_\xi Y, J_\xi e_i),
\end{align*}
and so $([\iota, \rJ] \, T)(\xi^{a+2}, Y^{b-1}) = - \|\xi\|^2 \sum\nolimits_{i=1}^{4n} (\nabla_{e_i} T)(\xi^a, Y^{b-1}, e_i)$. Furthermore,
\begin{align*}
  (\delta (\tau T)) & (\xi^a, Y^{b-1}) = (b-1)(\nabla_Y T)(\xi^a, Y^{b-2})  = (b-1) Y(T(\xi^a, Y^{b-2})) \\
   & = (b-1) \sum\nolimits_{i=1}^{4n} Y(T(\xi^a, Y^{b-2}, e_i, e_i)) = (b-1) \sum\nolimits_{i=1}^{4n} (\nabla_Y T)(\xi^a, Y^{b-2}, e_i, e_i),\\
  (\tau (\delta T)) & (\xi^a, Y^{b-1}) = \sum\nolimits_{i=1}^{4n} (\delta \, T)(\xi^a, Y^{b-1}, e_i, e_i)\\
   &= \sum\nolimits_{i=1}^{4n} ((b-1) (\nabla_Y T)(\xi^a, Y^{b-2}, e_i, e_i) + 2 (\nabla_{e_i} T)(\xi^a, Y^{b-1}, e_i)),
\end{align*}
from which $([\delta, \tau] \, T)(\xi^a, Y^{b-1}) = - 2 \sum\nolimits_{i=1}^{4n} (\nabla_{e_i} T)(\xi^a, Y^{b-1}, e_i)$, and~\eqref{eq:iJH} follows.

\subsubsection{Proof of identity~\eqref{eq:diH} of Lemma~\ref{l:bracketsH}}
\label{sss:iotadelta}

For $T \in \ocS^{a,b}$, we have $\iota T \in \ocS^{a+1,b-1}, \, \delta T \in \ocS^{a,b+1}$, and so $\iota\delta T, \delta\iota T \in \ocS^{a+1,b}$.

  As before, we assume that in a neighbourhood of a point $p \in S^{4n+3}$, all the vector fields below are chosen as in the last paragraph of Subsection~\ref{sss:point}.

  We compute at the point $p$ using~\eqref{eq:def4opH} and~\eqref{eq:nablaHatp}: 
  \begin{equation} \label{eq:diT1}
  \begin{split}
    (\delta(\iota T))(\xi^{a+1},Y^b) & = b (\nabla_Y (\iota T))(\xi^{a+1},Y^{b-1}) = b Y \big((\iota T)(\xi^{a+1},Y^{b-1})\big) \\
    &= b \sum\nolimits_{i=1}^{4n} Y \Big((\nabla_{e_i} T)(\xi^a, Y^{b-1},J_\xi e_i)\Big) = \Psi_1 + \Psi_2 + \Psi_3 + \Psi_4,
  \end{split}
  \end{equation}
  where
  \begin{equation} \label{eq:diT1terms}
  \begin{split}
    \Psi_1 &= b \sum\nolimits_{i=1}^{4n} (\nabla^2_{Y,e_i} T) (\xi^a, Y^{b-1},J_\xi e_i), \\
    \Psi_2 &= b \sum\nolimits_{i=1}^{4n} (\nabla_{\nabla_Y e_i}T)(\xi^a, Y^{b-1},J_\xi e_i), \\ 
    \Psi_3 &= a b \sum\nolimits_{i=1}^{4n} (\nabla_{e_i}T)(\xi^{a-1}, J_\xi Y, Y^{b-1},J_\xi e_i), \\ 
    \Psi_4 &= - b \sum\nolimits_{i=1}^{4n} \sum\nolimits_{\a=1}^{3} \<J_\a Y, J_\xi e_i\>(\nabla_{e_i}T)(\xi^a,\xi_\a, Y^{b-1}). 
  \end{split}
  \end{equation}
  By~\eqref{eq:nablaTam1}, for the expression $\Psi_3$ in~\eqref{eq:diT1terms} we get:
  \begin{multline*}
    (\nabla_{e_i}T)(\xi^{a-1}, J_\xi Y, Y^{b-1},J_\xi e_i) = \sum\nolimits_{\a=1}^{3} \big((b-1) \<J_\a e_i,Y\> T(\xi^{a-1}, \xi_\a,J_\xi Y, Y^{b-2},J_\xi e_i) \\
    + \<J_\a e_i, J_\xi Y\> T(\xi^{a-1}, \xi_\a, Y^{b-1},J_\xi e_i) + \<J_\a e_i,J_\xi e_i\> T(\xi^{a-1}, \xi_\a, J_\xi Y, Y^{b-1}) \big),
  \end{multline*}
  and so
  \begin{multline*}
    \Psi_3 = 4nab \, T(\xi^a, J_\xi Y, Y^{b-1}) \\
    - ab \sum\nolimits_{\a=1}^{3} \Big((b-1) T(\xi^{a-1}, \xi_\a,J_\xi Y, Y^{b-2},J_\xi J_\a Y) + T(\xi^{a-1}, \xi_\a, Y^{b-1},J_\xi J_\a J_\xi Y)\Big)\\
    =  (4n+2) ab \, T(\xi^a, J_\xi Y, Y^{b-1}) - ab \|\xi\|^2 \sum\nolimits_{\a=1}^{3}  T(\xi^{a-1}, \xi_\a, Y^{b-1}, J_\a Y) \\
    - ab(b-1) \sum\nolimits_{\a=1}^{3} T(\xi^{a-1}, \xi_\a,J_\xi Y, Y^{b-2},J_\xi J_\a Y),
  \end{multline*}
  where we have used the facts that $J_\xi J_\a J_\xi = -2 \<\xi,\xi_\a\> J_\xi + \|\xi\|^2 J_\a$.

  For $\Psi_4$, we have $(\nabla_{e_i}T)(\xi^a,\xi_\a, Y^{b-1}) = - a T(\xi^{a-1},\xi_\a, J_\xi e_i, Y^{b-1}) - T(\xi^a, J_\a e_i, Y^{b-1})$, by~\eqref{eq:nablaTap1}, which gives
  \begin{align*}
    \Psi_4 & = -b \sum\nolimits_{\a=1}^{3} \Big(a T(\xi^{a-1},\xi_\a, J_\xi^2 J_\a Y, Y^{b-1}) + T(\xi^a, J_\a J_\xi J_\a Y, Y^{b-1}\Big) \\
    & =  a b \|\xi\|^2 \sum\nolimits_{\a=1}^{3} T(\xi^{a-1},\xi_\a, J_\a Y, Y^{b-1}) - b T(\xi^a, J_\xi Y, Y^{b-1}),
  \end{align*}
  where we have used the facts that $J_\xi^2 = -\|\xi\|^2 I_{4n}$ and that $\sum\nolimits_{\a=1}^{3} J_\a J_\xi J_\a = J_\xi$.

  To compute $\Psi_2$ given by~\eqref{eq:diT1terms} we note that by~\eqref{eq:nablaHatp} we have $\nabla_Y e_i = - \sum\nolimits_{\a=1}^{3}\<J_\a Y, e_i\>\xi_\a$ at the point $p$ and that $\xi_\a(T(\xi^a, Y^{b-1},J_\xi e_i)) = 0$ by the $\Sp(1)$-invariance. As $\nabla_\eta Z = J_\eta Z$ from~\eqref{eq:nablaH1}, we obtain
  \begin{align*}
    \Psi_2 &=  b \sum\nolimits_{\a=1}^{3} \sum\nolimits_{i=1}^{4n} \<J_\a Y, e_i\> \Big(a T(\xi^{a-1}, \nabla_{\xi_\a} \xi, Y^{b-1},J_\xi e_i) \\
  & \hphantom{b \sum\nolimits_{\a=1}^{3} } +(b-1)T(\xi^a, Y^{b-2},J_\a Y,J_\xi e_i)+T(\xi^a, Y^{b-1},J_\a J_\xi e_i)\Big) \\
  & =  b\sum\nolimits_{\a=1}^{3} \Big(a T(\xi^{a-1}, \nabla_{\xi_\a} \xi, Y^{b-1},J_\xi J_\a Y) \\
  & \hphantom{b \sum\nolimits_{\a=1}^{3} } +(b-1)T(\xi^a, Y^{b-2},J_\a Y,J_\xi J_\a Y)+T(\xi^a, Y^{b-1},J_\a J_\xi J_\a Y)\Big)\\
  &= -a b \|\xi\|^2 \sum\nolimits_{\a=1}^{3} T(\xi^{a-1}, \xi_\a, Y^{b-1}, J_\a Y) + (a - b+2) b T(\xi^a, Y^{b-1},J_\xi Y),
  \end{align*}
  where in the last equation we have used the facts that $\sum\nolimits_{\a=1}^{3} J_\a J_\xi J_\a = J_\xi$ and that $\sum\nolimits_{\a=1}^{3} T(\xi^a, Y^{b-2},J_\a Y,J_\xi J_\a Y) = - T(\xi^a, Y^{b-1},J_\xi Y)$, and moreover, that by~\eqref{eq:nablaxixi} we have $\sum\nolimits_{\a=1}^{3} T(\xi^{a-1}, \nabla_{\xi_\a} \xi, Y^{b-1},J_\xi J_\a Y) = T(\xi^a, Y^{b-1},J_\xi Y)-\|\xi\|^2 \sum\nolimits_{\a=1}^{3} T(\xi^{a-1}, \xi_\a, Y^{b-1}, J_\a Y)$.

  Substituting the expressions for $\Psi_2, \Psi_3$ and $\Psi_4$ into~\eqref{eq:diT1} we obtain:
  \begin{equation} \label{eq:diT2}
  \begin{split}
    (\delta(\iota T))(\xi^{a+1} ,Y^b)= \Psi_1 & - a b (b-1) \sum\nolimits_{\a=1}^{3} T(\xi^{a-1}, \xi_\a, Y^{b-2}, J_\xi Y, J_\xi J_\a Y) \\
    & + b(a (4n +3) - b + 1) T(\xi^a, J_\xi Y, Y^{b-1})\\
    & - ab \|\xi\|^2 \sum\nolimits_{\a=1}^{3} T(\xi^{a-1},\xi_\a, J_\a Y, Y^{b-1}).
  \end{split}
  \end{equation}

  We further compute at the point $p$ by~\eqref{eq:def4opH} and~\eqref{eq:nablaHatp}: 
  \begin{equation} \label{eq:idT1}
  \begin{split}
    (\iota(\delta T))(\xi^{a+1},Y^b) & = \sum\nolimits_{i=1}^{4n} (\nabla_{e_i} (\delta T))(\xi^a, Y^b,J_\xi e_i) \\
    & = \sum\nolimits_{i=1}^{4n} e_i \big((\delta T)(\xi^a, Y^b,J_\xi e_i)) \\
    & = \sum\nolimits_{i=1}^{4n} e_i \big( b(\nabla_Y T)(\xi^a, Y^{b-1},J_\xi e_i) + (\nabla_{J_\xi e_i} T)(\xi^a, Y^b)\big) \\
    & = \Psi_5 + \Psi_6 + \Psi_7 + \Psi_8,
  \end{split}
  \end{equation}
  where
  \begin{equation} \label{eq:idT1terms}
  \begin{split}
    \Psi_5& = \sum\nolimits_{i=1}^{4n} \big( b (\nabla^2_{e_i,Y}T) (\xi^a, Y^{b-1},J_\xi e_i) + (\nabla^2_{e_i,J_\xi e_i}T)(\xi^a, Y^b) \big),\\
    \Psi_6& = b \sum\nolimits_{i=1}^{4n} (\nabla_{\nabla_{e_i}Y} T)(\xi^a, Y^{b-1},J_\xi e_i) + \sum\nolimits_{i=1}^{4n}(\nabla_{\nabla_{e_i} J_\xi e_i} T)(\xi^a, Y^b), \\ 
    \Psi_7& = \sum\nolimits_{i=1}^{4n} \big(ab(\nabla_Y T)(\xi^{a-1}, Y^{b-1},(J_\xi e_i)^2) + a (\nabla_{J_\xi e_i} T)(\xi^{a-1}, J_\xi e_i, Y^b) \big), \\ 
    \Psi_8& = - b \sum\nolimits_{i=1}^{4n} \sum\nolimits_{\a=1}^{3} \Big((b-1)\<J_\a e_i, Y\> (\nabla_Y T)(\xi^a, \xi_\a, Y^{b-2},J_\xi e_i) \\
    & \hphantom{=-b\sum_{i=1}^{4n}}+ \<J_\a e_i ,J_\xi e_i\> (\nabla_Y T)(\xi^a, \xi_\a, Y^{b-1}) + \<J_\a e_i, Y\> (\nabla_{J_\xi e_i} T)(\xi^a, \xi_\a, Y^{b-1})\Big). 
  \end{split}
  \end{equation}
  We simplify the summand on the right-hand side of $\Psi_7$ in~\eqref{eq:idT1terms} using~\eqref{eq:nablaTam1}:
  \begin{multline*}
    ab(\nabla_Y T)(\xi^{a-1}, Y^{b-1},(J_\xi e_i)^2) + a (\nabla_{J_\xi e_i} T)(\xi^{a-1}, J_\xi e_i, Y^b)\\
    = ab \sum\nolimits_{\a=1}^{3} \big(2\<J_\a Y,J_\xi e_i\> T(\xi^{a-1}, \xi_\a, Y^{b-1},J_\xi e_i) + \<J_\a J_\xi e_i, Y\> T(\xi^{a-1}, \xi_a, J_\xi e_i, Y^{b-1}) \big)\\
    = ab \sum\nolimits_{\a=1}^{3} \<J_\a Y,J_\xi e_i\> T(\xi^{a-1}, \xi_\a, Y^{b-1},J_\xi e_i),
  \end{multline*}
  and so $\Psi_7 = ab \|\xi\|^2 \sum\nolimits_{\a=1}^{3} T(\xi^{a-1}, \xi_\a, Y^{b-1},J_\a Y)$.

  Furthermore, by~\eqref{eq:nablaTap1} the summand on the right-hand side of $\Psi_8$ in~\eqref{eq:idT1terms} equals
  \begin{multline*}
    -\sum\nolimits_{\a=1}^{3} \big((b-1) \<J_\a e_i, Y\> ( a T(\xi^{a-1}, \xi_\a, Y^{b-2},J_\xi Y, J_\xi e_i) + T(\xi^a, Y^{b-2}, J_\a Y, J_\xi e_i)) \\
    + \<J_\a e_i, Y\> \big( - \|\xi\|^2 a T(\xi^{a-1}, \xi_\a, e_i, Y^{b-1}) + T(\xi^a, J_\a J_\xi e_i, Y^{b-1})\big)\\
    - (a+1) T (\xi^a, Y^{b-1}, J_\xi Y),
  \end{multline*}
  and so we obtain
  \begin{align*}
    \Psi_8 &= b \sum\nolimits_{\a=1}^{3} \Big((b-1)
     (-a T(\xi^{a-1}, \xi_\a, Y^{b-2},J_\xi Y, J_\xi J_\a Y) - T(\xi^a, Y^{b-2}, J_\a Y, J_\xi J_\a Y)) \\
    & \hphantom{b \sum\nolimits_{\a=1}^{3} \Big(} +\|\xi\|^2 a T(\xi^{a-1}, \xi_\a, J_\a Y, Y^{b-1}) - T(\xi^a, J_\a J_\xi J_\a Y, Y^{b-1})\Big)\\
    & \hphantom{= } + 4n b(a+1) T (\xi^a, Y^{b-1}, J_\xi Y)\\
    &= -a b \sum\nolimits_{\a=1}^{3} \Big((b-1) T(\xi^{a-1}, \xi_\a, Y^{b-2},J_\xi Y, J_\xi J_\a Y) -\|\xi\|^2 T(\xi^{a-1}, \xi_\a, J_\a Y, Y^{b-1}) \Big)\\
    & \hphantom{= }+ b( 4n (a+1) + b - 2) T (\xi^a, Y^{b-1}, J_\xi Y),
  \end{align*}
  using $\sum\nolimits_{\a=1}^{3} J_\a J_\xi J_\a = J_\xi$ and $\sum\nolimits_{\a=1}^{3} T(\xi^a, Y^{b-2}, J_\a Y, J_\xi J_\a Y) = -T(\xi^a, Y^{b-1}, J_\xi Y)$.

  Next we note that the first term on the right-hand side of $\Psi_6$ in~\eqref{eq:idT1terms} equals $-\Psi_2$, as at the point $p$, we have $\nabla_{e_i}Y = - \nabla_Y e_i$, by~\eqref{eq:nablaHatp}. In the second term of $\Psi_6$, we compute $\sum\nolimits_{i=1}^{4n} \nabla_{e_i} J_\xi e_i = \! - \! \sum\nolimits_{i=1}^{4n} \sum\nolimits_{\a=1}^{3} \<J_\a e_i, J_\xi e_i\> \xi_\a = -4n\xi$, and so it equals $4bn T(\xi^a, Y^{b-1},J_\xi Y)$, by~\eqref{eq:nablaH1} and~\eqref{eq:nablaxixi}, since $\xi(T(\xi^a, Y^b)) = 0$, by the $\Sp(1)$-invariance. Hence
  \begin{equation*}
  \Psi_6 = b(4n - a + b-2)) T(\xi^a, Y^{b-1},J_\xi Y) + a b \|\xi\|^2 \sum\nolimits_{\a=1}^{3} T(\xi^{a-1}, \xi_\a, Y^{b-1}, J_\a Y).
  \end{equation*}

  Substituting the expressions for $\Psi_6, \Psi_7$ and $\Psi_8$ into~\eqref{eq:idT1} we obtain:
  \begin{equation} \label{eq:idT2}
  \begin{split}
    (\iota(\delta T))(\xi^{a+1},Y^b) = \Psi_5 & - a b (b-1) \sum\nolimits_{\a=1}^{3} T(\xi^{a-1}, \xi_\a, Y^{b-2},J_\xi Y, J_\xi J_\a Y)\\
    &+ b( 4n (a+2) - a + 2b - 4) T (\xi^a, Y^{b-1}, J_\xi Y)\\
    &+ 3ab \|\xi\|^2 \sum\nolimits_{\a=1}^{3} T(\xi^{a-1}, \xi_\a, Y^{b-1},J_\a Y).
  \end{split}
  \end{equation}

  Finally, the expression $\Psi_1 - \Psi_5$ is the sum of the following two terms (where $R$ is the curvature tensor of $S^{4n+3}$): 
  \begin{multline*}
    b \sum\nolimits_{i=1}^{4n} ((\nabla^2_{Y,e_i} - \nabla^2_{e_i,Y}) T)(\xi^a, Y^{b-1},J_\xi e_i)  = b \sum\nolimits_{i=1}^{4n} (R(Y,e_i) T)(\xi^a, Y^{b-1},J_\xi e_i) \\
    = - b \sum\nolimits_{i=1}^{4n} ((b-1) T(\xi^a, Y^{b-2}, \<e_i,Y\>Y - \|Y\|^2 e_i,J_\xi e_i) + T(\xi^a, Y^{b-1}, -\<Y,J_\xi e_i\>e_i)) \\
    = - b^2 T(\xi^a, Y^{b-2}, Y ,J_\xi Y),
  \end{multline*}
  and (using the fact that the horizontal basis $\{\|\xi\|^{-1} J_\xi e_i\}$ is also orthonormal)
  \begin{multline*} 
    -\sum\nolimits_{i=1}^{4n} (\nabla^2_{e_i,J_\xi e_i} T)(\xi^a, Y^b) = -\tfrac12 \sum\nolimits_{i=1}^{4n} (R(e_i,J_\xi e_i)T)(\xi^a, Y^b) \\
    =  \tfrac12 b \sum\nolimits_{i=1}^{4n} T(\xi^a, Y^{b-1}, \<J_\xi e_i, Y\> e_i - \<e_i, Y\> J_\xi e_i) = - b T(\xi^a, Y^{b-1}, J_\xi Y ).
  \end{multline*}

  Thus subtracting~\eqref{eq:idT2} from~\eqref{eq:diT2} we obtain
  \begin{multline*}
    ((\delta\iota - \iota\delta) T)(\xi^{a+1} ,Y^b) = -4b(2n+b-a-1) T(\xi^a, J_\xi Y, Y^{b-1}) \\
    -4ab \|\xi\|^2 \sum\nolimits_{\a=1}^{3} T(\xi^{a-1}, \xi_\a, Y^{b-1},J_\a Y),
  \end{multline*}
  which is equivalent to~\eqref{eq:diH} in view of~\eqref{eq:def4opH} and~\eqref{eq:nuJH}.

  This completes the proof of Lemma~\ref{l:bracketsH}.

\end{document}